\documentclass[12pt]{article}
\usepackage[margin=1.25in]{geometry}

\usepackage{pgfplots}
\usepackage{amsmath}
\usepackage{amssymb}
\usepackage{algorithmic}
\usepackage{algorithm}
\usepackage{mathtools}
\usepackage{color}
\usepackage{tikz}
\usetikzlibrary{arrows,
	arrows.meta,
	bending,calc,patterns,angles,quotes}
\usepackage{float}

\makeatletter
\providecommand*{\cupdot}{%
	\mathbin{%
		\mathpalette\@cupdot{}%
	}%
}
\newcommand*{\@cupdot}[2]{%
	\ooalign{%
		$\m@th#1\cup$\cr
		\hidewidth$\m@th#1\cdot$\hidewidth
	}%
}
\makeatother

\newtheorem{theorem}{Theorem}

\newtheorem{lemma}{Lemma}
\newtheorem{conjecture}{Conjecture}
\newtheorem{proposition}{Proposition}
\newtheorem{remark}{Remark}
\newtheorem{definition}{Definition}

\def\endpf{{\ \hfill\hbox{\vrule width1.0ex height1.0ex}\parfillskip 0pt
	}}
	\newenvironment{proof}{\noindent{\bf Proof:}}{\endpf}
\begin{document}
\title{Useful stochastic bounds  
in time-varying queues with service and patience times having general joint distribution}

\author{Shreehari Anand Bodas \thanks{Bodas acknowledges the support of the European Union’s Horizon 2020 research and innovation
programme [Marie Skłodowska-Curie Grant Agreement 945045] and the NWO Gravitation project
NETWORKS [Grant 024.002.003].} \footnote{Korteweg-de Vries Institute for Mathematics; University of Amsterdam; 1098 XG Amsterdam; Netherlands. E-mail: s.a.bodas@uva.nl }
\and Royi Jacobovic \thanks{Jacobovic acknowledges the support of the Israel Science Foundation, grant \#3739/24.} \footnote{School of Mathematical Sciences, Tel-Aviv University, Tel-Aviv, Israel, 6997800, E-mail: royijacobo@tauex.tau.ac.il.}}
\maketitle
\begin{abstract}
Consider a first-come, first-served single server queue  with an initial workload $x>0$ and customers who arrive according to an inhomogeneous Poisson process with  rate function $\lambda:[0,\infty)\rightarrow[0,\lambda_h ]$ for some $\lambda_h>0$. For each $i\in\mathbb{N}$, let $S_i$ (resp., $Y_i$) be the service (resp., patience) time of the $i$'th customer and assume that $(S_1,Y_1),(S_2,Y_2),\ldots$ is an iid sequence of bivariate random vectors with non-negative coordinates. A customer joins if and only if his patience time is not less than his prospective waiting time (i.e., the left-limit of the workload process at his arrival epoch). Let $\tau(x)$ be the first time when the system becomes empty and let $N^*_\lambda(\cdot)$ be the arrival process of those who join the queue. In the present work we suggest a novel coupling technique which is applied to derive  stochastic upper bounds for the functionals: \begin{equation*}
    \int_0^{\tau(x)}g\circ W_x(t){\rm d}t\ \ \text{and}\ \ \int_0^{\tau(x)}g\circ W_x(t){\rm d}N^*_\lambda(t)\,,
\end{equation*} 
where $W_x(\cdot)$ is the workload process in the queue and $g(\cdot)$ is any lower semi-continuous function. We also demonstrate how to utilise these bounds under the additional assumption that $\lambda(\cdot)$ is periodic in order to study the asymptotics of the tail distribution of the cycle length, deriving stability conditions, integrability of the steady-state distribution, etc.
\end{abstract}

\section{Introduction}\label{sec: intro}
Queues are an integral part of everyone's daily life, e.g., traffic jams, waiting lines in a supermarket, internet traffic, etc. In standard observable queues  the willingness of the customers to join the queue is decreasing in the observed congestion level at their arrival epochs. Thus, depending on their patience,  some customers will join but others will balk. 

Let $\mathcal{L}$ be the set of all bivariate cdf's whose supports are subsets of $(0,\infty)\times(0,\infty]$ and a general model of a queue with impatient customers is provided in the next definition.
\begin{definition}\label{def: FCFS}
For each $x>0$, Borel function $\lambda:[0,\infty)\rightarrow[0,\infty)$ and $\Psi\in\mathcal{L}$, define the $\text{M}_t/\text{G}(\Psi)/1+\text{H}(\Psi)$ queue (with an initial workload $x$) as follows: 
\begin{enumerate}
    \item It is a single server queue with an infinite waiting room which is operated according to a first-come, first-served (FCFS) discipline.

    \item At time $t=0$, there is one initial customer in the queue whose remaining service time is equal to $x$.

    \item The arrival process  is an inhomogeneous Poisson process with rate $\lambda(\cdot)$.

    \item Let $\left(S_1,Y_1\right),\left(S_2,Y_2\right),\ldots$ be an iid sequence of bivariate random vectors having the cdf $\Psi$ and independent from the arrival process such that for every $i\in\mathbb{N}$:
    \begin{enumerate}
        \item The $i$'th customer joins the queue if and only if, at his arrival epoch, $Y_i$ is not less than his prospective waiting time (i.e., the left-limit of the workload process at his arrival epoch).
 
        \item If the $i$'th customer joins the queue, then his service time is equal to $S_i$.
    \end{enumerate}  
\end{enumerate}
\end{definition} 
For every $\Psi\in\mathcal{L}$, denote
\begin{equation*}
    G(\cdot)\equiv G(\Psi)(\cdot)\equiv \Psi(\cdot,\infty)\,,
\end{equation*}
and
\begin{equation*}
    \ H(\cdot)\equiv H(\Psi)(\cdot)\equiv \Psi(\infty,\cdot)\equiv\lim_{s\uparrow\infty}\Psi(s,\cdot)\,.
\end{equation*}
Note that in the context of Definition \ref{def: FCFS}, for each $i\in\mathbb{N}$,  $S_i$ (resp. $Y_i$) is the service (resp. patience) time of the $i$'th customer. Thus, since $(S_i,Y_i)$ is distributed according to $\Psi(\cdot)$ and $(S_1,Y_1),(S_2,Y_2),\ldots$ is an iid sequence, then $G(\cdot)$ (resp. $H(\cdot)$) is the service (resp. patience) distribution of the $\text{M}_t/\text{G}(\Psi)/1+\text{H}(\Psi)$ queue. In particular, when $\Psi(\cdot)$ has a product-form (i.e., $\Psi(s,y)\equiv G(s)H(y)$ for any $(s,y)\in(0,\infty)\times(0,\infty]$), this model is known in the literature as the $\text{M}_t/\text{G}/1+\text{H}$ queue (see the work by Bodas, Mandjes and Ravner \cite{Bodas2023} for a multi-server version of this model). If, in addition, $\lambda(\cdot)$ is a constant function, then this queue will become M/G/$1$+H queue which have already been studied by several authors, e.g., Baccelli, Boyer and Hebuterne \cite{Baccelli1984}, Bae and Kim \cite{Bae2001}, Perry and Asmussen \cite{Perry1995}, Boxma, Perry, and Stadje \cite{Boxma2011} and Boxma, Perry, Stadje and Zacks, \cite{Boxma2010}. 

A major drawback of many queueing models (including the well-studied $\text{M}/\text{G}/1+\text{H}$ queue) has been identified by Ross \cite{Ross1978}, who wrote:
\begin{quote}
``One of the major difficulties in attempting to apply known queueing theory
results to real problems is that almost always these results assume a time-stationary Poisson arrival process, whereas in practice the actual process is
almost invariably non-stationary."
\end{quote}
Another drawback follows from some studies which revealed an empirical evidence for existence of dependence between customer's service and patience times (see, e.g., the works by De Vries, Roy and De Koster \cite{De Vries2018} and Reich \cite{Reich2012}). 

By contrast, the present work is devoted to studying the $\text{M}_t/\text{G}(\Psi)/1+\text{H}(\Psi)$ queue which includes the more general assumption that the arrival process will be an \textit{inhomogeneous} Poisson process and $\Psi$ does \textit{not} necessarily have a product-form. The analysis will require a special technique which we have not seen elsewhere. The fact that we needed to develop a new methodology in order to analyse the $\text{M}_t/\text{G}(\Psi)/1+\text{H}(\Psi)$ queue is aligned with the following quote by Whitt \cite{Whitt2018}:   
\begin{quote}
    ``Service systems abound with queues, but the most natural direct models are often time-varying queues, which may require nonstandard analysis methods beyond stochastic textbooks." 
 \end{quote}  

\subsection{General description of the problem}
Fix $x>0$ and let $W_\lambda(\cdot;x)$ be the workload process in an $\text{M}_t/\text{G}(\Psi)/1+\text{H}(\Psi)$ queue with an initial workload $x$ and an arrival rate $\lambda(\cdot)$  for which there exists $\lambda_h>0$ such that $\lambda(t)\leq \lambda_h$ for every $t\geq0$. In addition, assume that $W_{\lambda_h}(\cdot;x)$ is the workload process in an $\text{M}/\text{G}(\Psi)/1+\text{H}(\Psi)$ queue with an initial workload $x$ and an arrival rate $\lambda_h$. Theorem 1 in the work by Lewis and Shedler \cite{Lewis1979} implies that whenever $\lambda(t)\leq\lambda_h $ for every $t\geq0$, there exists a probability space with two processes $N_{\lambda_h}(\cdot)$ and $N_\lambda(\cdot)$ such that: $N_{\lambda_h }(\cdot)$ is a homogeneous Poisson process with rate $\lambda_h $ and $N_\lambda(\cdot)$ is an inhomogeneous Poisson process with rate $\lambda(\cdot)$ which is a  thinning of $N_{\lambda_h }(\cdot)$. Therefore, as illustrated by the upper drawing of Figure \ref{fig: stochastic dominance}, when the customers are fully patient (i.e., when $H(y)=0$, for every $y\in\mathbb{R}$) we may couple the $\text{M}_t/\text{G}(\Psi)/1+\text{H}(\Psi)$ queue and the $\text{M}/\text{G}(\Psi)/1+\text{H}(\Psi)$ queue such that 
\begin{equation}\label{eq: stochastic dominance}
 W_{\lambda_h }(t;x)\geq W_{\lambda}(t;x)\ \ , \ \ \forall t,x\geq0\ \ , \ \ \mathbb{P}\text{-a.s.}   
\end{equation}
However, as illustrated by the lower drawing of Figure \ref{fig: stochastic dominance}, if the customers are allowed to have finite patience times, then \eqref{eq: stochastic dominance} becomes false. This observation leads to the following question: Is there some sense (weaker than the one in \eqref{eq: stochastic dominance}) in which $W_{\lambda}$ is stochastically dominated by $W_{\lambda_h}$? In the present work we provide an affirmative answer for this query by specifying a general class of functionals $\mathcal{H}\equiv\{h:\mathcal{D}([0,\infty))\rightarrow[0,\infty];h\in\mathcal{H}\}$ such that 
\begin{equation}\label{eq: functionals}
    h\left[W_\lambda(\cdot;x)\right]\leq_{\normalfont \text{st}}h\left[W_{\lambda_h}(\cdot;x)\right]\ \ , \ \ \forall x\geq0\ , \ h\in\mathcal{H}\,,
\end{equation}
where $\mathcal{D}([0,\infty))$ is the Skorokhod space of all non-negative c\'adl\'ag functions  defined on $[0,\infty)$ and $\leq_{\normalfont \text{st}}$ stands for a first-order stochastic dominance  between the distributions of two random variables (see, e.g., the book by Shaked and Shanthikumar \cite{Shaked2007}).
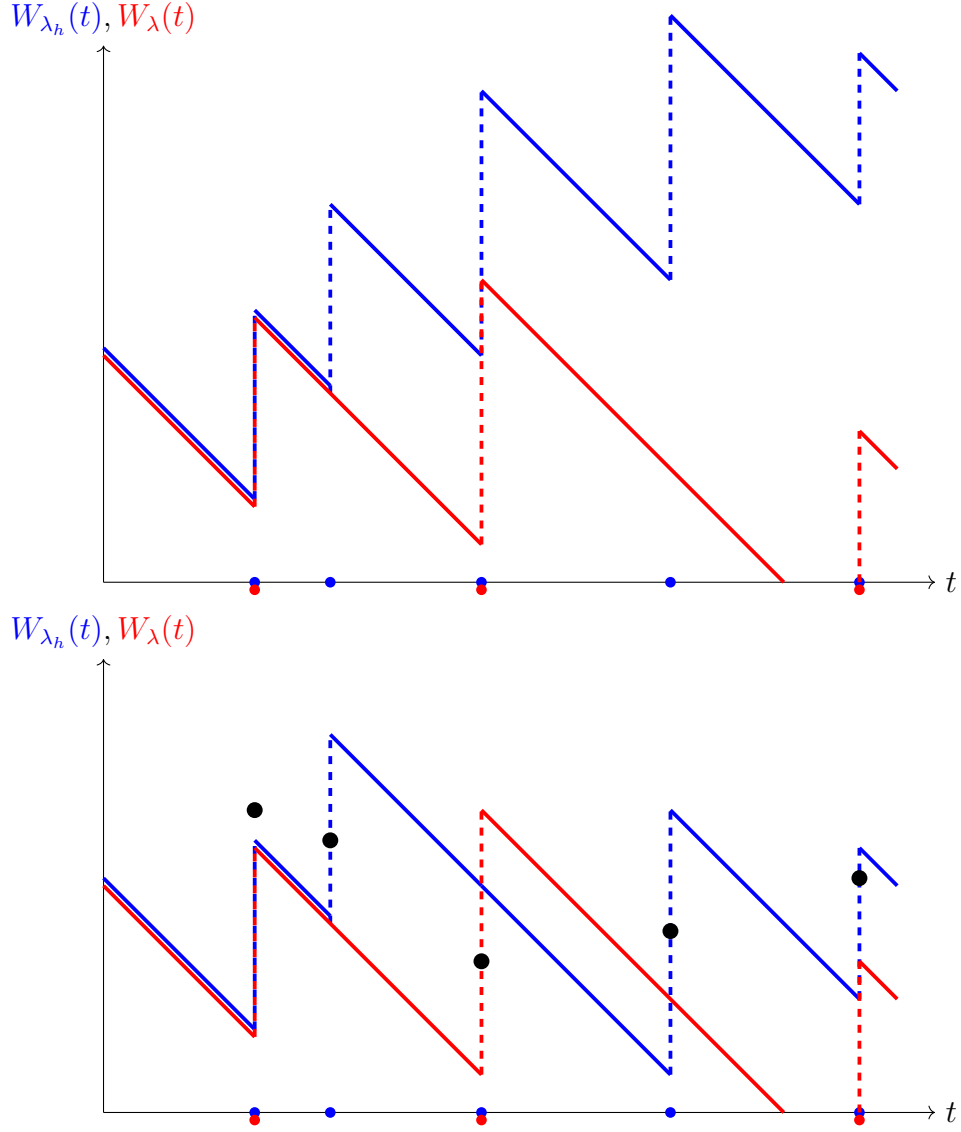
\begin{figure}\centering
\begin{tikzpicture}
        \draw[->] (0,0) -- (11,0) node[right] {$t$};
        \draw[->] (0,0) -- (0,7.1) node[above] {$\textcolor{blue}{W_{\lambda_h}(t)}, \textcolor{red}{W_\lambda(t)}$};

        % red and blue

        \fill[color = blue] (canvas cs: x = 2cm, y = 0cm) circle(2pt);
        \fill[color = red] (canvas cs: x = 2cm, y = -0.1cm) circle(2pt);
        
        \fill[color = blue] (canvas cs: x = 3cm,y = 0cm) circle (2pt);
        
        \fill[color = blue] (canvas cs: x = 5cm, y = 0cm) circle (2pt);
        \fill[color = red] (canvas cs: x = 5cm, y = -0.1cm) circle(2pt);
        
        \fill[color = blue] (canvas cs: x= 7.5cm, y = 0cm) circle (2pt);
        
        \fill[color = blue] (canvas cs: x = 10cm, y = 0cm) circle (2pt);
        \fill[color = red] (canvas cs: x = 10cm, y = -0.1cm) circle(2pt);

        \draw[line width=0.5mm, color = blue]{} (0cm, 3.1cm) -- (2cm, 1.1cm);
        \draw[line width=0.5mm, color = blue, dashed]{} (2cm, 1.1cm) -- (2cm, 3.6cm);
        \draw[line width=0.5mm, color = blue]{} (2cm, 3.6cm) -- (3cm, 2.6cm);
        \draw[line width=0.5mm, color = blue, dashed]{} (3cm, 2.5cm) -- (3cm, 5cm);
        \draw[line width=0.5mm, color = blue]{} (3cm, 5cm) -- (5cm, 3cm);
        \draw[line width = 0.5mm, color = blue, dashed]{}{} (5cm, 3cm) -- (5cm, 6.5cm);
        \draw[line width=0.5mm, color = blue]{} (5cm, 6.5cm) -- (7.5cm, 4cm);
        \draw[line width=0.5mm, color = blue, dashed]{} (7.5cm, 4cm) -- (7.5cm, 7.5cm);
        \draw[line width=0.5mm, color = blue]{} (7.5cm, 7.5cm) -- (10cm, 5cm);
        \draw[line width=0.5mm, color = blue, dashed]{}{} (10cm, 5cm) -- (10cm, 7cm);
        \draw[line width=0.5mm, color = blue]{}{} (10cm, 7cm) -- (10.5cm, 6.5cm);

        \draw[line width=0.5mm, color = red]{} (0cm, 3cm) -- (2cm, 1cm);
        \draw[line width=0.5mm, color = red, dashed]{} (2cm, 1cm) -- (2cm, 3.5cm);
        \draw[line width=0.5mm, color = red]{} (2cm, 3.5cm) -- (5cm, 0.5cm);
        \draw[line width=0.5mm, color = red, dashed]{} (5cm, 0.5cm) -- (5cm, 4 cm);
        \draw[line width=0.5mm, color = red]{}{} (5cm, 4cm) -- (9cm, 0cm);
        % \draw[line width=0.5mm, color = red]{}{} (9cm, 0cm) -- (10cm, 0cm);
        \draw[line width=0.5mm, color = red, dashed]{}{} (10cm, 0cm) -- (10cm, 2cm);
        \draw[line width=0.5mm, color = red]{}{} (10cm, 2cm) -- (10.5cm, 1.5cm);
    \end{tikzpicture}
    \begin{tikzpicture}
        \draw[->] (0,0) -- (11,0) node[right] {$t$};
        \draw[->] (0,0) -- (0,6) node[above] {$\textcolor{blue}{W_{\lambda_h}(t)}, \textcolor{red}{W_\lambda(t)}$};

        \fill[color = blue] (canvas cs: x = 2cm, y = 0cm) circle(2pt);
        \fill[color = red] (canvas cs: x = 2cm, y = -0.1cm) circle(2pt);
        
        \fill[color = blue] (canvas cs: x = 3cm,y = 0cm) circle (2pt);
        
        \fill[color = blue] (canvas cs: x = 5cm, y = 0cm) circle (2pt);
        \fill[color = red] (canvas cs: x = 5cm, y = -0.1cm) circle(2pt);
        
        \fill[color = blue] (canvas cs: x= 7.5cm, y = 0cm) circle (2pt);
        
        \fill[color = blue] (canvas cs: x = 10cm, y = 0cm) circle (2pt);
        \fill[color = red] (canvas cs: x = 10cm, y = -0.1cm) circle(2pt);

        \draw[line width=0.5mm, color = blue]{} (0cm, 3.1cm) -- (2cm, 1.1cm);
        \draw[line width=0.5mm, color = blue, dashed]{} (2cm, 1.1cm) -- (2cm, 3.6cm);
        \draw[line width=0.5mm, color = blue]{} (2cm, 3.6cm) -- (3cm, 2.6cm);
        \draw[line width=0.5mm, color = blue, dashed]{} (3cm, 2.5cm) -- (3cm, 5cm);
        \draw[line width=0.5mm, color = blue]{} (3cm, 5cm) -- (5cm, 3cm);
        \draw[line width=0.5mm, color = blue]{} (5cm, 3cm) -- (7.5cm, 0.5cm);
        \draw[line width=0.5mm, color = blue, dashed]{} (7.5cm, 0.5cm) -- (7.5cm, 4cm);
        \draw[line width=0.5mm, color = blue]{} (7.5cm, 4cm) -- (10cm, 1.5cm);
        \draw[line width=0.5mm, color = blue, dashed]{}{} (10cm, 1.5cm) -- (10cm, 3.5cm);
        \draw[line width=0.5mm, color = blue]{}{} (10cm, 3.5cm) -- (10.5cm, 3cm);

        \draw[line width=0.5mm, color = red]{} (0cm, 3cm) -- (2cm, 1cm);
        \draw[line width=0.5mm, color = red, dashed]{} (2cm, 1cm) -- (2cm, 3.5cm);
        \draw[line width=0.5mm, color = red]{} (2cm, 3.5cm) -- (5cm, 0.5cm);
        \draw[line width=0.5mm, color = red, dashed]{} (5cm, 0.5cm) -- (5cm, 4cm);
        \draw[line width=0.5mm, color = red]{}{} (5cm, 4cm) -- (9cm, 0cm);
        % \draw[line width=0.5mm, color = red]{}{} (9cm, 0cm) -- (10cm, 0cm);
        \draw[line width=0.5mm, color = red, dashed]{}{} (10cm, 0cm) -- (10cm, 2cm);
        \draw[line width=0.5mm, color = red]{}{} (10cm, 2cm) -- (10.5cm, 1.5cm);

        \fill[color = black] (canvas cs: x = 2cm, y = 4cm) circle(3pt);
        \fill[color = black] (canvas cs: x = 3cm, y = 3.6 cm) circle(3pt);
        \fill[color = black] (canvas cs: x = 5cm, y = 2cm) circle(3pt);
        \fill[color = black] (canvas cs: x = 7.5cm, y = 2.4cm) circle(3pt);
        \fill[color = black] (canvas cs: x = 10cm, y = 3.1cm) circle(3pt);
    \end{tikzpicture}
        \caption{In both drawings: (1) The blue graph is an illustration of the workload process in an $\text{M}/\text{G}(\Psi)/1+\text{H}(\Psi)$ queue with an arrival rate $\lambda_h $. (2) The red graph is an illustration of the workload process in an $\text{M}_t/\text{G}(\Psi)/1+\text{H}(\Psi)$ queue in which the arrival process is a thinning of the arrival process in the $\text{M}/\text{G}(\Psi)/1+\text{H}(\Psi)$ queue described by the blue graph. (3) Each arrival to the $\text{M}/\text{G}(\Psi)/1+\text{H}(\Psi)$ (resp., $\text{M}_t/\text{G}(\Psi)/1+\text{H}(\Psi)$) queue is denoted by a blue (resp., red) dot on the horizontal axis at its occurrence epoch. The upper drawing illustrates that when all customers have infinite patience times, the blue graph will never go below the red graph. The lower drawing illustrates that when the customers have finite patience times (noted by thick black circles), the red graph may be above the blue graph and vice-versa.} \label{fig: stochastic dominance}
    \end{figure}

\subsection{The main results}
Consider an $\text{M}_t/\text{G}(\Psi)/1+\text{H}(\Psi)$ queue with an arrival rate function $\lambda(\cdot)$  and an initial workload $x>0$.
Denote the workload process in the queue by $W_{\lambda}(\cdot;x)\equiv W_{\lambda}(\cdot;x,\Psi)$. In addition, consider an $\text{M}/\text{G}(\Psi)/1+\text{H}(\Psi)$ queue with an arrival rate $\lambda_h>0$ and an initial workload $x>0$. Denote the workload process in the second queue by $W_{\lambda_h }(\cdot;x)\equiv W_{\lambda_h }(\cdot;x,\Psi)$. For both of these queues, define the first time that they become empty, i.e.,
\begin{align}\label{eq: tau}
&\tau_\lambda(x)\equiv\tau_\lambda\left(x;\Psi\right)\equiv\inf\left\{t\geq0;W_\lambda\left(t;x,\Psi\right)=0\right\},\\& \nonumber\tau_{\lambda_h }(x)\equiv\tau_{\lambda_h }\left(x;\Psi\right)\equiv\inf\left\{t\geq0;W_{\lambda_h }\left(t;x,\Psi\right)=0\right\}.
\end{align}
Now, take some  Borel function $g:[0,\infty)\rightarrow[0,\infty)$ and introduce the resulting workload functionals by:
    \begin{align*}
    &A_\lambda(g)\equiv A_\lambda(g;x)\equiv A_\lambda(g;x,\Psi)\equiv\int_0^{\tau_\lambda(x;\Psi)}g\circ W_\lambda(t;x,\Psi){\rm d}t\,,\\&A_{\lambda_h }(g)\equiv A_{\lambda_h }(g;x) \equiv A_{\lambda_h }(g;x,\Psi)\equiv\int_0^{\tau_{\lambda_h }(x;\Psi)}g\circ W_{\lambda_h }(t;x,\Psi){\rm d}t\,.\nonumber
\end{align*}
The economic interpretation of $g(\cdot)$ is as follows: For every $w\in[0,\infty)$, $g(w)$ refers to the holding cost of $w$ workload units per infinitesimal unit of time. Thus, $A_\lambda(g;x)$ refers to the total holding cost which is accumulated during the period from time $t=0$, when the queue is initialized with $x$ workload units, and until the queue becomes empty.

For the statement of the main results, define
\begin{equation}\label{eq: G}
    \mathcal{G}\equiv\left\{g:[0,\infty)\rightarrow[0,\infty);\exists\left(g_n\right)_{n=1}^\infty\subseteq\mathcal{C}_b\left([0,\infty)\right)\ \text{s.t.}\ 0\leq g_n\uparrow g\ \text{as}\ n\uparrow\infty\right\},
\end{equation}
 i.e., the set of all non-negative lower semi-continuous functions on $[0,\infty)$. The following Theorems \ref{thm: main1} and \ref{thm: main2} are the main results of this work, as they jointly construct a general class of functionals for which \eqref{eq: functionals} is valid.
\begin{theorem}\label{thm: main1}
If $\lambda(t)\leq \lambda_h$ for all $t\geq0$, then
\begin{equation*}
  A_\lambda(g;x,\Psi)\leq_{\normalfont \text{st}}A_{\lambda_h }(g;x,\Psi),  \ \forall x>0, \ g\in\mathcal{G}, \ \Psi\in\mathcal{L}.
\end{equation*}  
\end{theorem}
Let $N^*_\lambda(\cdot)$ (resp., $N^*_{\lambda_h}(\cdot)$) be the \textit{effective} arrival process in the $\text{M}_t/\text{G}(\Psi)/1+\text{H}(\Psi)$ (resp., $\text{M}/\text{G}(\Psi)/1+\text{H}(\Psi)$) queue, i.e., for each $t\geq0$, $N^*_\lambda(t)$ (resp., $N^*_{\lambda_h}(t)$) is equal to the number of customers who \textit{joined} the $\text{M}_t/\text{G}(\Psi)/1+\text{H}(\Psi)$ (resp., $\text{M}/\text{G}(\Psi)/1+\text{H}(\Psi)$) queue before or at time $t$. Then, take some  Borel function $g:[0,\infty)\rightarrow[0,\infty)$ and introduce the corresponding workload functionals: 
    \begin{align}\label{eq: A definition}
    &A^*_\lambda(g)\equiv A^*_\lambda(g;x)\equiv A^*_\lambda(g;x,\Psi)\equiv\int_0^{\tau_\lambda(x;\Psi)}g\circ W_\lambda(t-;x,\Psi){\rm d}N^*_\lambda(t)\,,\\&A^*_{\lambda_h }(g)\equiv A^*_{\lambda_h }(g;x) \equiv A^*_{\lambda_h }(g;x,\Psi)\equiv\int_0^{\tau_{\lambda_h }(x;\Psi)}g\circ W_{\lambda_h }(t-;x,\Psi){\rm d}N^*_{\lambda_h}(t)\,.\nonumber
\end{align}
Note that once there are customers who balk, the effective arrival process takes into account only those who join the queue, i.e., it is different from the standard arrival process (denoted by $N_\lambda(\cdot)$ and $N_{\lambda_h}(\cdot)$ for the $\text{M}_t/\text{G}(\Psi)/1+\text{H}(\Psi)$ and $\text{M}/\text{G}(\Psi)/1+\text{H}(\Psi)$ queues, respectively), which counts also those who balk. In addition, when the integration is carried out with respect to the effective arrival process, the economic interpretation of $g(\cdot)$ in the resulting workload functional reads: For every $w\geq0$,   $g(w)$ is the cost of letting a new customer to join the queue when the existing workload just before his arrival epoch is equal to $w$. 

\begin{theorem}\label{thm: main2}
If $\lambda(t)\leq \lambda_h$ for every $t\geq0$, then 
\begin{equation*}
    A^*_\lambda(g;x,\Psi)\leq_{\normalfont \text{st}}A^*_{\lambda_h }(g;x,\Psi), \  \forall x>0, \ g\in\mathcal{G}, \ \Psi\in\mathcal{L}\,.
\end{equation*} \end{theorem} 
In particular, notice that the stochastic dominance relations which are stated in Theorem \ref{thm: main1} and Theorem \ref{thm: main2} are valid for every $x>0$. Thus, once the initial state $x$ is distributed over $(0,\infty)$ according to a non-degenerate initial distribution $\pi(\cdot)$, then Theorem \ref{thm: main1} (resp., \ref{thm: main2}) implies that for any bounded nondecreasing function $f:[0,\infty)\rightarrow\mathbb{R}$ we have that
    \begin{equation}\label{eq: initial distribution}
        \int \mathbb{E}_xf\left[A_\lambda(g;x,\Psi)\right]\pi\left({\rm d}x\right)\leq 
        \int \mathbb{E}_xf\left[A_{\lambda_h}(g;x,\Psi)\right]\pi\left({\rm d}x\right), \ \forall g\in\mathcal{G}, \ \Psi\in\mathcal{L}
    \end{equation}
    \begin{equation*}
        \left(\text{resp., } 
        \int \mathbb{E}_xf\left[A^*_\lambda(g;x,\Psi)\right]\pi\left({\rm d}x\right)\leq 
        \int \mathbb{E}_xf\left[A^*_{\lambda_h}(g;x,\Psi)\right]\pi\left({\rm d}x\right), \  \forall g\in\mathcal{G}, \ \Psi\in\mathcal{L}\right),
    \end{equation*}
    where $\mathbb{E}_x$ stands for the conditional expectation with respect to $\mathbb{P}$ given the initial state $x$ (which is distributed according to the initial distribution $\pi(\cdot)$). 
    Thus, by conditioning and un-conditioning on $x$, the statements of Theorem \ref{thm: main1} and Theorem \ref{thm: main2} will remain valid for any $g\in\mathcal{G}$, $\Psi\in\mathcal{L}$ and an initial distribution $\pi(\cdot)$ over $(0,\infty)$. 
\subsection{The Driving Force Behind the Proofs}
This section includes a general description of the main benchmarks in the proofs of the main results. In addition, we provide some explanations of these benchmarks and discuss how they jointly yield the main results. The first step in the proof of Theorem \ref{thm: main1} is to define the following auxiliary queue:
\begin{definition}\label{def: LCFS}
For each $k\in\overline{\mathbb{Z}}_+\equiv\mathbb{Z}_{\geq0}\cup\{\infty\}$, each $x>0$, each Borel function $\lambda:[0,\infty)\rightarrow[0,\infty)$ and each $\Psi\in\mathcal{L}$, define the  {\normalfont  $\text{M}_t/\text{G}(\Psi)/1/k+\text{H}(\Psi)$} queue (with an initial workload $x$) as follows: 
\begin{enumerate}
    \item It is a single server queue with a waiting room of size $k$ {\normalfont(}excluding the one in service{\normalfont)} which is operated according to a preemptive last-come, first-served (LCFS-PR) discipline. 

    \item At time $t=0$, there is one initial customer in the queue whose remaining service time is equal to $x$.

    \item The arrival process  is an inhomogeneous Poisson process with rate $\lambda(\cdot)$.

    \item Let $\big(S_1^{(k)},Y_1^{(k)}\big),\big(S_2^{(k)},Y_2^{(k)}\big),\ldots$ be an iid sequence of bivariate random vectors distributed according to $\Psi$ and independent from the arrival process such that for every $i\in\mathbb{N}$:
    \begin{enumerate}
        \item The $i$'th customer joins the queue if and only if, the next two conditions are satisfied: 
        \begin{itemize}
            \item There is an empty place for him, i.e., the left limit of the queue-length process {\normalfont(}counting also the potential customer in the service position{\normalfont)} at his arrival epoch is less than $k+1$.

            \item $Y_i^{(k)}$ is not less than the left-limit of the workload process at his arrival epoch.

        \end{itemize}  
        \item If the $i$'th customer joins the queue, then his service time is equal to $S_i^{(k)}$.
    \end{enumerate}  
\end{enumerate}
\end{definition}
\begin{remark}
    \normalfont Notably, as stated in Definition \ref{def: FCFS}, the service discipline in an $\text{M}_t/\text{G}(\Psi)/1+\text{H}(\Psi)$ queue is FCFS while, as stated in Definition \ref{def: LCFS}, for any $k\in\overline{\mathbb{Z}}_+$, the service discipline in an $\text{M}_t/\text{G}(\Psi)/1/k+\text{H}(\Psi)$ is LCFS-PR. 
\end{remark}

For each $i\in\mathbb{N}$, we will say that the `patience time' of the $i$'th customer is equal to $Y_i^{(k)}$. Note that the `patience' times in this  $\text{M}_t/\text{G}(\Psi)/1/k+\text{H}(\Psi)$ queue are not really  representing the \textit{patience} of the customers. Since this class of queues is just a benchmark towards the analysis of the $\text{M}_t/\text{G}(\Psi)/1+\text{H}(\Psi)$ queue, we prefer to keep calling them `patience times' in analogy with the original $\text{M}_t/\text{G}(\Psi)/1+\text{H}(\Psi)$ queue. In addition, observe that, by construction, the workload process of the  $\text{M}_t/\text{G}(\Psi)/1/\infty+\text{H}(\Psi)$ queue has the same distribution as the workload process of the $\text{M}_t/\text{G}(\Psi)/1+\text{H}(\Psi)$ queue. 

Now, for each $k\in\overline{\mathbb{Z}}_+$ and $x>0$, denote the workload process in the  $\text{M}_t/\text{G}(\Psi)/1/k+\text{H}(\Psi)$ queue with an initial workload $x$ by $W^{(k)}_{\lambda}(\cdot;x)\equiv W^{(k)}_{\lambda}(\cdot;x,\Psi)$. Then, define the analogues of $\tau_\lambda(x),\tau_{\lambda_h}(x),A_\lambda(g)$ and $A_{\lambda_h}(g)$ in that queue, i.e.,
\begin{align*}
&\tau^{(k)}_\lambda(x)\equiv\tau^{(k)}_\lambda\left(x;\Psi\right)\equiv\inf\left\{t\geq0;W^{(k)}_\lambda\left(t;x,\Psi\right)=0\right\},\\& \nonumber\tau^{(k)}_{\lambda_h }(x)\equiv\tau^{(k)}_{\lambda_h }\left(x;\Psi\right)\equiv\inf\left\{t\geq0;W^{(k)}_{\lambda_h }\left(t;x,\Psi\right)=0\right\},
\end{align*}
and also for any Borel function $g:[0,\infty)\rightarrow[0,\infty)$
    \begin{align*}
    &A^{(k)}_\lambda(g)\equiv A^{(k)}_\lambda(g;x)\equiv A^{(k)}_\lambda(g;x,\Psi)\equiv\int_0^{\tau^{(k)}_\lambda(x;\Psi)}g\circ W^{(k)}_\lambda(t;x,\Psi){\rm d}t\,,\\&A^{(k)}_{\lambda_h }(g)\equiv A^{(k)}_{\lambda_h }(g;x) \equiv A^{(k)}_{\lambda_h }(g;x,\Psi)\equiv\int_0^{\tau^{(k)}_{\lambda_h }(x;\Psi)}g\circ W^{(k)}_{\lambda_h }(t;x,\Psi){\rm d}t\,.\nonumber
\end{align*}
The following lemma is the key to the proof of the main results.
\begin{proposition}\label{prop: LCFS-PR}
Assume that  $\lambda(t)\leq\lambda_h$ for every $t\geq0$. Then, for every Borel function $g:[0,\infty)\rightarrow[0,\infty)$, $x>0$, $\Psi\in\mathcal{L}$ and $k\in\mathbb{Z}_{\geq0}$, the stochastic dominance $A^{(k)}_\lambda\left(g;x,\Psi\right)\leq_{\normalfont \text{st}}A^{(k)}_{\lambda_h }\left(g;x,\Psi\right)$ holds and $A^{(k)}_{\lambda_h }\left(g;x,\Psi\right)$ is finite $\mathbb{P}$-a.s.     
\end{proposition}
The proof of Proposition \ref{prop: LCFS-PR} uses an induction on $k$. Namely, the initial case $k=0$ is trivial. Then, the induction step will follow via a coupling technique that we will introduce later.  Some idea of how the induction hypothesis will be applied may be described as follows: Consider an $\text{M}_t/\text{G}(\Psi)/1/(k+1)+\text{H}(\Psi)$ queue with an initial customer whose service time is equal to $x$. Observe that each customer who interrupts the service of the initial customer initiates a busy period of an $\text{M}_t/\text{G}\big(\widetilde{\Psi}\big)/1/k+\text{H}\big(\widetilde{\Psi}\big)$ queue which starts at his arrival epoch and lasts until he leaves the queue. Importantly, in that $\text{M}_t/\text{G}\big(\widetilde{\Psi}\big)/1/k+\text{H}\big(\widetilde{\Psi}\big)$ queue, the arrival rate is equal to a random translation of $\lambda$ depending on the arrival epoch of the customer who initiated this excursion. In addition, $\widetilde{\Psi}$ is a random element in $\mathcal{L}$ which depends on the state of the queue at the arrival epoch of the customer who initiates that busy period.

The challenge stems from the fact that there might be several interruptions to the service of the initial customer, leading to several busy periods of queues having different arrival rate functions and different service-patience times joint distributions. In the proof we will explain how to exactly couple them properly in order to make Proposition \ref{prop: LCFS-PR} follow from the induction hypothesis.   

The proof of Theorem \ref{thm: main1} will be done by applying the stochastic dominance guaranteed by Proposition \ref{prop: LCFS-PR} for any finite $k$ and taking the limit $k\to\infty$ carefully. In that step we will need to apply the assumption that $g(\cdot)$ is lower semi-continuous. 

A high-level explanation why this methodology works out is as follows: To begin with, fix $x>0$ and $\Psi\in\mathcal{L}$. Then, consider all  $\text{M}_t/\text{G}(\Psi)/1/k+\text{H}(\Psi)$ queues (for any $k\in\overline{\mathbb{Z}}_+$) in the same probability space. We set them such that they share the same arrival process, service times and patience times, i.e., the only perturbation in their definitions is due to the differences in their waiting room sizes.  Then, on any \textit{fixed} finite time-window, the workload process of the  $\text{M}_t/\text{G}(\Psi)/1/k+\text{H}(\Psi)$ queue converges as $k\to\infty$, $\mathbb{P}$-a.s. to the workload process of the  $\text{M}_t/\text{G}(\Psi)/1/\infty+\text{H}(\Psi)$ queue. Consequently, as explained before, the workload process of the  $\text{M}_t/\text{G}(\Psi)/1/\infty+\text{H}(\Psi)$ queue has the same distribution as the workload process of the $\text{M}_t/\text{G}(\Psi)/1+\text{H}(\Psi)$ queue and Theorem \ref{thm: main1} eventually follows. 

\subsection{Periodic $\text{M}_t/\text{G}(\Psi)/1+\text{H}(\Psi)$ queues and applications}\label{subsec: applications}
In this part we introduce a consequence of Theorems \ref{thm: main1} and \ref{thm: main2} in an $\text{M}_t/\text{G}(\Psi)/1+\text{H}(\Psi)$ queue with periodic arrival rate. Namely, 
consider an $\text{M}_t/\text{G}(\Psi)/1+\text{H}(\Psi)$ queue which is empty at time $t=0$. Assume that the arrival rate in the queue $\lambda(\cdot)$ has a period $\kappa>0$, i.e., 
$\lambda(t)=\lambda(t+\kappa)$ for every $t\geq0$. Due to this assumption, $W_\lambda(\cdot)\equiv W_\lambda(\cdot;0)$ has a regenerative epoch at 
\begin{equation*}
    \xi_\lambda\equiv\inf\left\{\kappa n;\  W_\lambda(\kappa n)=0, \ n\in\mathbb{N}\right\}\,.
\end{equation*}
The theory of regenerative processes (see, e.g., Chapter IV of the book by Asmussen \cite{Asmussen2003}) suggests studying the distribution of functionals over cycles, i.e.,
\begin{equation}\label{eq: form1}
   \mathcal{A}_\lambda(g)\equiv \mathcal{A}_\lambda(g;\Psi)\equiv\int_0^{\xi_\lambda}g\circ W_\lambda(t){\rm d}t\,,
\end{equation}
or
\begin{equation*}
\mathcal{A}^*_\lambda(g)\equiv \mathcal{A}^*_\lambda(g;\Psi)\equiv\int_0^{\xi_\lambda}g\circ W_\lambda(t-){\rm d}N_\lambda^*(t)\,.
\end{equation*}
In the sequel, we will prove that the main results stated in Theorems \ref{thm: main1} and \ref{thm: main2} jointly imply the next Theorem \ref{thm: super stochastic bound} which includes upper stochastic bounds for $\mathcal{A}_\lambda(g)$ and $\mathcal{A}^*_\lambda(g)$. In practice, the upper stochastic bounds stated in Theorem \ref{thm: super stochastic bound} are relatively tractable since they are both geometric decompoundings (translated by the $\max\{1,\kappa\}$) of the distributions of the analogue workload functionals (over busy periods) associated with queue having a homogeneous arrival process.

For what follows, for any $n\in\mathbb{N}$ and a univariate cdf $F(\cdot)$, let $F^{\star n}(\cdot)$ be its $n$-fold convolution. Also, for any random variable $Z$, we will denote its cdf by $\text{dist}(Z)$. 
\begin{theorem}\label{thm: super stochastic bound}
    Assume that $g\in\mathcal{G}$, $\Psi\in\mathcal{L}$ and $\lambda(\cdot)$ has a period $\kappa>0$ such that $\lambda(t)\leq\lambda_h\in(0,\infty)$ for every $t\geq0$. In addition, let
    \begin{align*}
        &F_{\lambda_h}(\cdot)\equiv F_{\lambda_h}(\cdot;g,\Psi)\equiv{\normalfont\text{dist}}\left[A_{\lambda_h}(g;x,\Psi)\right],\\&\nonumber F^*_{\lambda_h}(\cdot)\equiv F^*_{\lambda_h}(\cdot;g,\Psi)\equiv{\normalfont\text{dist}}\left[A_{\lambda_h}^*(g;x,\Psi)\right],
    \end{align*}  
    when the initial distribution of $x$ is $G(\cdot)\equiv G(\Psi)(\cdot)$. In addition,   
    denote the next two geometric decompoundings of $F_{\lambda_h}(\cdot)$ and $F^*_{\lambda_h}(\cdot)$:
    \begin{align*}
    &J_{\lambda_h}(u;g,\Psi)\equiv\sum_{n=1}^\infty\left(1-e^{-\kappa\lambda_h}\right)^{n-1}e^{-\kappa\lambda_h}F^{\star n}_{\lambda_h}(u-\max\{1,\kappa\}), \ \forall u\in\mathbb{R}\,,\\& J^*_{\lambda_h}(u;g,\Psi)\equiv\sum_{n=1}^\infty\left(1-e^{-\kappa\lambda_h}\right)^{n-1}e^{-\kappa\lambda_h}\left(F^*_{\lambda_h}\right)^{\star n}(u-\max\{1,\kappa\}), \ \forall u\in\mathbb{R}\,.  \nonumber 
    \end{align*}
    Then, 
    \begin{equation*}{\normalfont\text{dist}}\left[\mathcal{A}_\lambda(g;\Psi)\right]\leq_{\normalfont \text{st}}J_{\lambda_h}(\cdot;g,\Psi)\ \ \text{and}\ \ {\normalfont\text{dist}}\left[\mathcal{A}^*_\lambda(g;\Psi)\right]\leq_{\normalfont \text{st}}J^*_{\lambda_h}(\cdot;g,\Psi)\,. 
    \end{equation*}
\end{theorem}
Some applications of Theorem \ref{thm: super stochastic bound} are given below:\newline\newline 
\textbf{Application 1 (tail asymptotics of $\xi_\lambda$):} In the sequel we will show that Theorem \ref{thm: super stochastic bound} implies the following upper bound on the rate at which $\mathbb{P}\left\{\xi_\lambda>u\right\}$ vanishes as $u\to\infty$.
\begin{theorem}\label{thm: heavy-tail}    Assume that $\rho_h\equiv\lambda_h\int_0^\infty y{\rm d}G(y)<1$ and let $\varphi$ be the length of the first busy period  in a {\normalfont FCFS M/G/1} queue with an arrival rate $\lambda_h$ and an initial workload equal to zero. If $G(\cdot)\equiv G(\Psi)(\cdot)$ is subexponential such that 
\begin{equation}\label{eq: heavy tail}
\lim_{u\to\infty}\frac{\mathbb{P}\left\{\varphi>u\right\}}{1-G\left[u(1-\rho_h)\right]}=(1-\rho_h)^{-1},
\end{equation}
then 
    \begin{equation*}
     \lim_{u\to\infty}\sup\frac{\mathbb{P}\left\{\xi_\lambda>u\right\}}{1-G\left\{[u-g(0)\max\{\kappa,1\}](1-\rho_h)\right\}}\leq \frac{1-e^{-\kappa\lambda_h}}{(1-\rho_h)e^{-\kappa\lambda_h}}\,.  
    \end{equation*}
\end{theorem}
\begin{remark}
    \normalfont De Meyer and Teugels \cite{De Meyer1980} proved that \eqref{eq: heavy tail} is valid whenever $G(\cdot)$ is regularly-varying. Later, Zwart \cite{Zwart2001} extended this result to the class of all intermediate regularly-varying. It turns out that \eqref{eq: heavy tail} is valid also when the tail of $G(\cdot)$ is heavier than $e^{-\sqrt{u}}$ (for the precise details see the works by Jelenkovi\'c, and Mom\v cilovi\'c \cite{Jelenkovic2004} and Baltr\=unas, Daley and Kl\"uppelberg \cite{Baltrunas2004}). Another related work is by Asmussen, Kl\"uppelberg and Sigman \cite{Asmussen1999}, who explained why \eqref{eq: heavy tail} cannot hold when the tail of $G(\cdot)$ is lighter than $e^{-\sqrt{u}}$.
\end{remark}
\textbf{Application 2 (tail asymptotics of $\int_0^{\xi_\lambda}{\rm d}N_\lambda^*(t)$):} Let $\eta^*_\lambda$ be the number of customers who get service during the first cycle of the $\text{M}_t/\text{G}(\Psi)/1+\text{H}(\Psi)$ queue, i.e.,  $\eta^*_\lambda\equiv\int_0^{\xi_\lambda}{\rm d}N^*_\lambda(t)$. Then, in the sequel we will show that Theorem \ref{thm: super stochastic bound} implies the next result concerning an upper bound for the rate at which $\mathbb{P}\left\{\eta_\lambda^*>u\right\}$ vanishes as $u\to\infty$.
\begin{theorem}\label{thm: heavy-tail11}
    Assume that $\rho_h<1$ and denote the number of customers who get service during the first busy period in a {\normalfont FCFS M/G/1} queue by $\eta$.  If $G(\cdot)\equiv G(\Psi)(\cdot)$ is subexponential such that 
    \begin{equation}\label{eq: heavy tail11}
    \lim_{u\to\infty}\frac{\mathbb{P}\left\{\eta>u\right\}}{1-G\left[\frac{u(1-\rho_h)}{\lambda_h}\right]}=(1-\rho_h)^{-1}\,, 
    \end{equation}
    then, 
    \begin{equation*}
     \lim_{u\to\infty}\sup\frac{\mathbb{P}\left\{\eta^*_\lambda>u\right\}}{1-G\left\{\frac{\left[u-g(0)\max\{\kappa,1\}\right](1-\rho_h)}{\lambda_h}\right\}}\leq \frac{1-e^{-\kappa\lambda_h}}{(1-\rho_h)e^{-\kappa\lambda_h}}\,.  
    \end{equation*}
\end{theorem}
\begin{remark}
    \normalfont Some conditions on $G(\cdot)$ which imply \eqref{eq: heavy tail11} appear in the works by Baltr\=unas, Daley and Kl\"uppelberg \cite{Baltrunas2004}, Denisov and Shneer \cite{Denisov2010} and Doney \cite{Doney1989} (basically $G(\cdot)$ should be either intermediate regularly-varying or have a tail which is heavier than $e^{-\sqrt{u}}$ with some additional moment conditions).
\end{remark}

\begin{remark}\normalfont
    Note that Theorem \ref{thm: heavy-tail} as well as Theorem \ref{thm: heavy-tail11} are valid for any $\Psi\in\mathcal{L}$, i.e., even when $\Psi$ does not have a product-form.
\end{remark}
\textbf{Application 3 (stability condition):} Let $\tau_{\lambda_h}$ be a random variable having the distribution of $\tau_{\lambda_h}(x)$ (see, \eqref{eq: tau}) with an initial state $x$ which is distributed according to the service distribution $G(\cdot)$. When $g\equiv1$, Theorem \ref{thm: super stochastic bound} implies that $\mathbb{E}\xi_\lambda<\infty$ whenever $\mathbb{E}\tau_{\lambda_h}<\infty$. Equivalently, this means that whenever the M/$\text{G}(\Psi)$/1+$\text{H}(\Psi)$ queue (with an arrival rate $\lambda_h$) is stable, the $\text{M}_t/\text{G}(\Psi)/1+\text{H}(\Psi)$ queue is also stable. Baccelli, Boyer and Hebuterne \cite{Baccelli1984} proved that under the assumption that $\Psi(\cdot)$ has a product-form, the stability condition of the dominating M/G/1+H queue is 
\begin{equation}\label{eq: stability conditionnn}
   \lambda_h\int_0^\infty y{\rm d}G(y)\lim_{y\uparrow\infty}\left[1-H(y)\right]<1\,.
\end{equation}
Note that as $0\cdot\infty$ is undefined, this condition implicitly implies that $G(\cdot)$ has a first finite moment. 

Note that the fact that the $\text{M}_t/\text{G}(\Psi)/1+\text{H}(\Psi)$ queue is stable under the assumption that $\Psi(\cdot)$ has a product-form and \eqref{eq: stability conditionnn} is valid can also be proven by applying Lyapunov functions; see, e.g., Theorem 4.1 (and its proof) in the recent work by Bodas, Mandjes, and Ravner \cite{Bodas2023}. 
\newline\newline 
\textbf{Application 4 (integrability of the steady-state distribution):} Assume that the $\text{M}_t/\text{G}(\Psi)/1+\text{H}(\Psi)$ queue is stable, e.g., when $\Psi(\cdot)$ has a product-form and \eqref{eq: stability conditionnn} is valid. Consequently, since the cycle length distribution is nonarithmetic, there exists a steady-state (limiting) workload distribution, i.e., a random variable $W_\lambda(\infty)$ such that $W_\lambda(t)\xrightarrow{{\normalfont\text{d}}}W_\lambda(\infty)$ as $t\to\infty$. It is known that by setting $g(w)\equiv w$, we get
\begin{equation*}
    \mathbb{E}W_\lambda(\infty)=\frac{\mathbb{E}\mathcal{A}_\lambda(g)}{\mathbb{E}\xi_\lambda}\,,
\end{equation*}
and we want to derive sufficient conditions under which $\mathbb{E}W_\lambda(\infty)<\infty$. 

Applying Theorem \ref{thm: super stochastic bound} with $g(w)\equiv w$ yields that it is sufficient to derive conditions under which $\int_0^{\tau_{\lambda_h}}W_{\lambda_h}(t){\rm d}t$ has finite mean. When $\Psi$ has a product-form, Baccelli, Boyer and Hebuterne \cite{Baccelli1984} showed that, $\mathbb{E}\int_0^{\tau_{\lambda_h}}W_{\lambda_h}(t){\rm d}t<\infty$ whenever the stability condition \eqref{eq: stability conditionnn} and $\int_0^\infty y^2{\rm d}G(y)<\infty$ are both satisfied. 
\newline\newline
\textbf{Application 5 (finite moments of $\eta^*_\lambda$):}  Motivated by certain estimation problem,   Bodas, Mandjes and Ravner \cite{Bodas2023} assumed that $\Psi(\cdot)$ had a product-form and they looked for conditions on the primitives of the model under which $\eta_\lambda^*$ has a finite second moment. The following theorem is a consequence of Theorem \ref{thm: super stochastic bound} and it provides an answer for a more general question.  
\begin{theorem}\label{thm: eta}
    Fix $m\in\mathbb{N}$ and assume that: {\normalfont (i)} $\Psi$ has product-form. {\normalfont (ii)} The stability condition \eqref{eq: stability conditionnn} is satisfied.  
    \begin{enumerate}
        \item If $m=1$, then $\mathbb{E}\eta^*_\lambda<\infty$.

        \item If $m\geq2$ and  $\int_0^\infty y^m{\rm d}G(y)<\infty$, then $\mathbb{E}\left(\eta^*_\lambda\right)^m<\infty$.
        \end{enumerate} 
\end{theorem}
In particular, Theorem \ref{thm: eta} is useful when $g(w)\leq P(w)$ for every $w\geq0$ for some polynomial $P(\cdot)$. Namely, let $p$ be the degree of $P(\cdot)$ and observe that $W_{\lambda_h}(t-)\leq \xi_\lambda$ for every $t\in[0,\xi_\lambda]$. Thus, whenever $\eta^*_\lambda$ has a finite $(kp+1)$'th moment, $\mathcal{A}^*_\lambda(g)$ has a finite $k$'th moment. Conditions under which $\eta^*_\lambda$ has a finite $m\equiv(kp +1)$'th moment are supplied in the statement of Theorem \ref{thm: eta}.

\subsection{Related literature}\label{subsec: literature}
\textbf{Dependence between service and patience times:} Empirical evidence for the existence of dependence between service times and patience times of customers were provided by De Vries, Roy and De Koster \cite{De Vries2018} for restaurants and by  Reich \cite{Reich2012} for contact centers. Motivated by this empirical evidence, Wu, Bassamboo and Perry \cite{Wu2019} studied the impact of dependence between the service and patience times of the customers on different performance metrics of the queue, e.g., expected waiting times and average queue length, etc. They claimed in their research that the dependence structure makes the exact analysis intractable and hence their analysis is based on a stationary fluid approximation and simulation experiments which take into account the entire joint distribution of the service and patience times. Another work by Wu, Bassamboo and Perry \cite{Wu2022} includes a statistical procedure to estimate the dependence between the patience and service times. Yu and Perry \cite{Yu2023} referred to heavy-traffic limits in the case when both the service and patience times are exponentially distributed and perfectly correlated. Lastly, in a recent short communication, Moyal and Perry put forth an open question regarding the existence of a functional weak law of large numbers for the queue process, and a corresponding weak law of large numbers (WLLN) for the stationary distribution in the general case. 

Two conceptual differences between the present work and the above-mentioned strand of literature are: (1) To the best of our knowledge, our work is the first one to study a queue with both \textit{dependent} service and patience times as well as \textit{inhomogeneous}-time arrival process. (2) In the above-mentioned literature, each customer is forced to join the queue and he leaves the system either when his patience time has already expired or when he enters the service position before his service time has expired. 
\newline\newline 
\textbf{M/G/1+H queue:}
The stability condition of the M/G/1+H queue can be found in the work by Baccelli, Boyer and Hebuterne \cite{Baccelli1984}. The density of the stationary virtual-waiting time distribution in this queue for different choices of patience distributions was derived by several authors, including, Baccelli, Boyer and Hebuterne \cite{Baccelli1984}, Bae and Kim \cite{Bae2001} and Perry and Asmussen \cite{Perry1995}. The steady-state distribution of the number of customers in the M/G/1+H queue was studied by Boxma, Perry, and Stadje \cite{Boxma2011}. The distribution of the length of the busy period in the M/G/1+H queue was analysed by Boxma, Perry, Stadje and Zacks, \cite{Boxma2010}. Brandt and Brandt \cite{Brandt2013} suggested a unified approach leading to an extension of the results from the previous works by Boxma \textit{et al.} \cite{Boxma2010,Boxma2011}. A generalization of the the M/G/1+H queue in a different direction (by considering batch arrivals) was studied by Inoue, Boxma, Perry and Zacks \cite{Inoue2018}. Recently, Inoue, Ravner and Mandjes \cite{Inoue2023} derived certain asymptotic properties of the maximum-likelihood  estimator in a parametric estimation problem defined in a multi-server version of the M/G/1+H queue.
\newline\newline\textbf{Queues with periodic arrival rates:} To the best of our knowledge, the only existing work regarding a queue with balking customers and periodic arrival rate is a recent paper by Bodas, Mandjes and Ravner \cite{Bodas2023} who studied  some asymptotic properties of the maximum-likelihood  estimator in a parametric estimation problem defined in a multi-server version of that model. Another related work is by Massey and Whitt, \cite{Massey1994} who studied a multi-server version of that model with no waiting-room, i.e., if a customer arrives when all servers are occupied, he will immediately balk. Other works about time-varying queues with no balking and periodic arrival rates are by, e.g., Harrison and Lemoine \cite{Harrison1977}, Lemoine \cite{Lemoine1989}, Rolski \cite{Rolski1987,Rolski1989} and Whitt \cite{Whitt2014}. 
\newline\newline \textbf{Workload functionals:}
Recently, Glynn, Jacobovic and Mandjes  \cite{Glynn2023} studied the high-order joint moments of workload functionals in L\'evy-driven queues when $g(\cdot)$ is a polynomial. It is also conventional to consider $g(w)\equiv e^{-\alpha w}$ for some $\alpha>0$. This choice is usually needed in order to derive the Laplace-Stieltjes transform of the steady-state distribution of the workload process in queues with a homogeneous arrival process, see, e.g., the works by Kella and Whitt \cite{Kella1991, Kella1992}. 

Workload functionals that are defined via an integration with respect to the arrival process have been studied so far in the context of queueing models with a homogeneous arrival processes and no balking customers. Some examples of such works are, e.g., by Glynn \cite{Glynn1994}, Glynn, Jacobovic and Mandjes \cite{Glynn2023}, Jacobovic, Levering and Boxma \cite{Jacobovic2023a} and Jacobovic and Mandjes \cite{Jacobovic2024}.\newline\newline
\textbf{Stochastic dominance in queueing models:}
Bhattacharya and Ephremides \cite{Bhattacharya1991} derived a first-order stochastic dominance relation between certain workload functionals in a certain class of multi-server queues with impatient customers who arrive according to a \textit{homogeneous}-time process. Another relevant work is by Jouini and Dallery \cite{Jouini2007} in which there are some results regarding first-order stochastic dominance relation between certain workload functionals in an M/M/k/K+M queue, i.e., when there are $k$ servers, waiting room of size $K\in\mathbb{N}$ and $G(\cdot)$ as well as $H(\cdot)$ are exponential distributions (not necessarily with the same rate). Other results about stochastic monotonicity of functionals in queueing networks were derived by Shanthikumar and Yao \cite{Shanthikumar1986,Shanthikumar1987,Shanthikumar1989}. Results regarding stochastic monotonicity of functionals in tandem queues appear in the works by Adan and Van der Wal \cite{Adan1989} and Van Dijk and Van der Wal \cite{Van Dijk1989}. In addition, the seminal paper by Ross \cite{Ross1978} initiated another branch of literature about existence of an increasing-convex stochastic order between certain functionals of queues with an inhomogeneous arrival processes. Some prominent works in this direction are, e.g., by Heyman \cite{Heyman1982} and Rolski \cite{Rolski1981, Rolski1986}. Note that to the best of our knowledge, there are no subsequent works in this direction regarding queues with balking, stressing a solid line between these works and the present one. 
\subsection{Structure of the work}
Section \ref{sec: proof main results} includes the detailed proofs of the main results (Theorems \ref{thm: main1} and \ref{thm: main2}). Section \ref{sec: proof super bound} is devoted for the proof (based on the main results) of Theorem \ref{thm: super stochastic bound}. In Section \ref{sec: heavy tails} one finds the proof of Theorem \ref{thm: heavy-tail} and Theorem \ref{thm: heavy-tail11}. The proof of Theorem \ref{thm: eta} is given in Section \ref{sec: proof eta}. Lastly, Section \ref{sec: discussion} includes a discussion about the contents of the present work, with some optional directions for future research.

\section{Proofs of Theorems \ref{thm: main1} and \ref{thm: main2}}\label{sec: proof main results}
In this section we introduce the proofs of the main results of the present work (Theorem \ref{thm: main1} and Theorem \ref{thm: main2}). Since the proofs of these two theorems are essentially identical (up to a few minor modifications), we will only provide a detailed proof of Theorem~\ref{thm: main1}. 

\subsection{Proof of Proposition \ref{prop: LCFS-PR}}\label{sec: proposition}
As announced in Section 1.3, we use an inductive argument.%, but first, we need to define some extra notation. Then, we will supply the details of the induction.\newline\newline
%\textbf{Notations:} For each $k\in\mathbb{Z}_{\geq0}$, $x>0$ and $\Psi\in\mathcal{L}$, let $W^{(k)}_{\lambda}(\cdot)\equiv W^{(k)}_{\lambda}(\ \cdot\ ;x,\Psi)$ be the workload process in an $\text{M}_t/\text{G}(\Psi)/1/k+\text{H}(\Psi)$ queue with an arrival rate $\lambda(\cdot)$ and an initial customer with a remaining service time that equals $x$ at time $t=0$. In addition, 
%let $W^{(k)}_{\lambda_h }(\cdot)\equiv W_{\lambda_h }(\ \cdot\ ;x,\Psi)$ be the workload process in an $\text{M}/\text{G}(\Psi)/1/k+\text{H}(\Psi)$ queue with an arrival rate $\lambda_h$ and an initial customer with a remaining service time that equals $x$ at time $t=0$. Also, let $\tau^{(k)}_\lambda\equiv\tau^{(k)}_\lambda(x;\Psi)$ (resp., $\tau^{(k)}_{\lambda_h }\equiv\tau^{(k)}_{\lambda_h }(x;\Psi)$) be the first time that $W^{(k)}_\lambda(\cdot)$ (resp., $W^{(k)}_{\lambda_h }(\cdot)$) hits the origin \textcolor{blue}{(not origin, but zero)}.  Then, for any nondecreasing function $g:[0,\infty)\rightarrow[0,\infty)$ denote 
%\begin{align}
 %   &A^{(k)}_\lambda(g)\equiv A^{(k)}_\lambda(g;x,\Psi)\equiv\int_0^{\tau^{(k)}_\lambda(x;\Psi)}g\circ W^{(k)}_\lambda(t;x,\Psi){\rm d}t\,,\\&A^{(k)}_{\lambda_h }(g)\equiv A^{(k)}_{\lambda_h }(g;x,\Psi)\equiv\int_0^{\tau^{(k)}_{\lambda_h }(x;\Psi)}g\circ W^{(k)}_{\lambda_h}(t;x,\Psi){\rm d}t\,.\nonumber
%\end{align}
\subsubsection*{The initial step:} Observe that $\tau_\lambda^{(0)}(x;\Psi)=\tau_{\lambda_h }^{(0)}(x;\Psi)=x$ and also
    \begin{equation*}
    W_{\lambda}^{(0)}(t;x,\Psi)=W_{\lambda_h }^{(0)}(t;x,\Psi)=x-t\ \ , \ \ \forall \ t\geq0\,, x>0\ , \ \Psi\in\mathcal{L}\,, 
    \end{equation*}
    i.e., the result holds true for $k=0$.

\subsubsection*{The induction step:}
Consider an $\text{M}_t/\text{G}(\Psi)/1/k+\text{H}(\Psi)$ queue with an arrival rate $\lambda(\cdot)$ and assume that Proposition \ref{prop: LCFS-PR} is true for $k\in\mathbb{Z}_{\geq0}$. Denote the (random) number of customers who interrupt the service of the initial customer by $\eta$. Without loss of generality, assume that the underlying probability space of this model maintains an iid sequence $U_1,U_2,\ldots$ of random variables which are uniformly distributed on $[0,1]$ and independent of  the queueing model. 

For each $1\leq i\leq \eta$ we will use the following notations:

\begin{enumerate}
        \item Let $c_i$ be the $i$-th customer who is interrupting the service of the initial customer.

        \item Let $T_i$ be the arrival time of $c_i$. 

        % Shreehari: Should S_i be remaining service time of initial customer?
        \item Let $S_i$ be the remaining service time of the initial customer at the arrival time of $c_i$. In particular, this means that  $S_1> S_2>\cdots>S_\eta$, $\mathbb{P}$-a.s.

        \item Let $I_i$ be the period when $c_i$ is present in the queue, i.e., $I_i$ is the period between the arrival and departure times of $c_i$. 

        \item Let $D_i$ be the service time of $c_i$.
        
    \end{enumerate}
         Then, Definition \ref{def: LCFS} yields
         \begin{equation}\label{eq: decomposition111}
             A^{(k+1)}_\lambda(g;x,\Psi)=\int_0^xg\left(x-t\right){\rm d}t+\sum_{i=1}^\eta\int_{I_i}g\circ W_\lambda^{(k+1)}(t;x,\Psi){\rm d}t
         \end{equation}
         where 
        \begin{equation*}
             \int_{I_i}g\circ W_\lambda^{(k+1)}(t;x,\Psi){\rm d}t\stackrel{{\normalfont \text{d}}}{=}A^{(k)}_{\lambda(T_i+\cdot)}\left[g(\cdot+S_i);D_i,\Psi\left(\cdot,\cdot+S_i\right)\right]
         \end{equation*}
        for $1\leq i\leq \eta$. Note that \eqref{eq: decomposition111} is based on the fact that the initial customer receives all his service requirement with probability one. To explain this fact, recall that for each $1\leq i\leq\eta$, $c_i$ initiates a busy period of an $\text{M}_t/\text{G}\big(\widetilde{\Psi}_i\big)/1/k+\text{H}\big(\widetilde{\Psi}_i\big)$ queue (for some $\widetilde{\Psi}_i$ to be specified later) that starts at his arrival epoch and lasts until he leaves the queue. Thus, the induction hypothesis (with $g\equiv1$) implies that $I_1,I_2,\ldots,I_\eta<\infty$, $\mathbb{P}$-a.s. Therefore, since $\eta<\infty$, $\mathbb{P}$-a.s., the initial customer gets all his service requirement with probability one justifying the first term (i.e., the deterministic integral) in the RHS of \eqref{eq: decomposition111}.
        
        Now, for each $1\leq i\leq \eta$ fix values $(d,s,t)\in\mathcal{S}\equiv\text{Supp}\left(D_i,S_i,T_i\right)$ and let $F(\cdot;d,s,t)$ (resp., $G(\cdot;d,s)$) be the cdf of 
        \begin{equation*}
         A^{(k)}_{\lambda(t+\cdot)}\left[g(\cdot+s);d,\Psi\left(\cdot,\cdot+s\right)\right]\ \ \left(\text{resp., } A^{(k)}_{\lambda_h }\left[g(\cdot+s);d,\Psi\left(\cdot,\cdot+s\right)\right]\right).  
        \end{equation*}
       For each $i\in\mathbb{N}$ define two new random variables 
        \begin{equation*}
            V_i(d,s,t)\equiv F^{-1}\left(U_i;d,s,t\right), \qquad Z_i(d,s)\equiv G^{-1}\left(U_i;d,s\right),
        \end{equation*}
        where $F^{-1}\left(\cdot;d,s,t\right)$ (resp., $G^{-1}\left(\cdot;d,s\right)$) is the pseudo-inverse of $F\left(\cdot;d,s,t\right)$ (resp., $G\left(\cdot;d,s\right)$). The induction hypothesis yields that for each $(d,s,t)\in\mathcal{S}$, $V_i(d,s,t)\leq Z_i(d,s)$, $\mathbb{P}$-a.s. Now, for each $i\in\mathbb{N}$ define two new random variables
        \begin{equation*}
          V_i\equiv\begin{cases} 
       V_i\left(D_i,S_i,T_i\right), & 1\leq i\leq \eta, \\
      0, & \text{otherwise},  
   \end{cases}\qquad \ Z_i\equiv \begin{cases} 
       Z_i\left(D_i,S_i\right), & 1\leq i\leq \eta, \\
      0 & \text{otherwise},  
   \end{cases}   
        \end{equation*}
        and recall that $U_1,U_2,\ldots$ is an iid sequence which is independent from the queueing model. %Thus, the collection of random variables
        %\begin{equation}
          % \left\{\left(V(d,s,t),Z(d,s)\right);(d,s,t)\in\mathcal{S}\right\} 
        %\end{equation} 
        %and $(D_i,S_i,T_i)_{i=1}^\eta$ are independent. 
        Consequently, standard conditioning and un-conditioning implies that 
        \begin{align*}
        \mathbb{P}\left\{\sum_{i=1}^\eta V_i\leq\sum_{i=1}^\eta Z_i\right\}&=\mathbb{E}\mathbb{P}\left\{\sum_{i=1}^\eta  V_i\left(D_i,S_i,T_i\right)\leq\sum_{i=1}^\eta  Z_i\left(D_i,S_i\right)\bigg|\left(D_i,S_i,T_i\right)_{i=1}^\eta\right\}\nonumber\\&=\mathbb{E}1=1\,.
        \end{align*}
        In addition, we know that 
            \begin{equation*}
                \sum_{i=1}^\eta V_i\stackrel{\text{d}}{=}\sum_{i=1}^\eta\int_{I_i}g\circ W_\lambda^{(k+1)}(t;x,\Psi){\rm d}t
            \end{equation*}
        and so
        \begin{equation*}
            \sum_{i=1}^\eta\int_{I_i}g\circ W_\lambda^{(k+1)}(t;x,\Psi){\rm d}t\leq_{\normalfont \text{st}}\sum_{i=1}^\eta Z_i\,,
        \end{equation*}
        such that conditionally on $\left(D_i,S_i\right)_{i=1}^\eta$, the random variables $\left(Z_i\right)_{i=1}^\eta$ are independent. 
       
        Now, by applying Theorem 1 from the work by Lewis and Shedler \cite{Lewis1979}, one can construct a probability space with: (1) A homogeneous Poisson process with rate $\lambda_h$ denoted by $N_{\lambda_h }(\cdot)$. (2) An inhomogeneous Poisson process with rate $\lambda(\cdot)$ denoted by $N_\lambda(\cdot)$ which is a thinning of $N_{\lambda_h}(\cdot)$. (3) An $\text{M}_t/\text{G}(\Psi)/1/(k+1)+\text{H}(\Psi)$ queue whose arrival process  is $N_\lambda(\cdot)$.
        
        We shall couple that queue with an $\text{M}/\text{G}(\Psi)/1/(k+1)+\text{H}(\Psi)$ queue with an arrival rate $\lambda_h $ as follows: First, assume that during the \textit{service periods} of the initial customer, the stream of customers to the queue is determined by $N_{\lambda_h }(\cdot)$. In particular, those customers which are shared by the two processes $N_{\lambda}(\cdot)$ and $N_{\lambda_h }(\cdot)$ have the same  patience times and service times as in the original preemptive  $\text{M}_t/\text{G}(\Psi)/1/(k+1)+\text{H}(\Psi)$ system. In addition, during the \textit{waiting periods} of the initial customer, the arrivals of the customers are governed by an external independent homogeneous  Poisson process with rate $\lambda_h $. Importantly, the patience times and the service times of the customers who are arriving according to the external process are iid with joint distribution  $\Psi(\cdot)$, such that they are independent from all other random variables in the model. Observe that the memoryless property of the exponential distribution ensures that the resulting queue is  $\text{M}/\text{G}(\Psi)/1/(k+1)+\text{H}(\Psi)$ with an arrival rate $\lambda_h $ as needed.

        Importantly, in that system, for each $1\leq i\leq \eta$, the customer $c_i$ arrives when the remaining service time of the initial customer equals $S_i$, just like in the original  $\text{M}_t/\text{G}(\Psi)/1/k+\text{H}(\Psi)$ queue. Therefore, as in the original queue, he joins and adds 
        \begin{equation*}
            \int_{I_i(\lambda_h )}g\circ W_{\lambda_h }^{(k+1)}(t;x,\Psi){\rm d}t\stackrel{\normalfont \text{d}}{=}A^{(k)}_{\lambda_h }\left[g(\cdot+S_i);D_i,\Psi\left(\cdot,\cdot+S_i\right)\right]
        \end{equation*}
        %\textcolor{blue}{(Should we add a point mentioning that there may be customers arriving according to $N_{\lambda_h}(.)$ which are not present in $N_\lambda(.)$ who also initiate a busy period?)}
        to the value of the objective functional, where $I_i(\lambda_h )$ is the period in which $c_i$ is present in the  $\text{M}/\text{G}(\Psi)/1/(k+1)+\text{H}(\Psi)$ queue with an arrival rate $\lambda_h$. In particular, observe that given $\left(S_n,D_n\right)_{n=1}^\eta$, these contributions are independent. Therefore, since $g(\cdot)$ is non-negative, then taking the sum of these values contributed by $c_1,c_2,\ldots,c_\eta$, one obtaining
        \begin{align*}
        A^{(k+1)}_{\lambda_h }(g;x,\Psi)-\int_{0}^{x} g(x-t) \mathrm{d}t &\geq\sum_{i=1}^\eta\int_{I_i(\lambda_h )}g\circ W_{\lambda_h }^{(k+1)}(t;x,\Psi){\rm d}t\\&\stackrel{\text{d}}{=}\sum_{n=1}^{\eta}Z_n\geq_{\text{st}}\sum_{i=1}^\eta\int_{I_i}g\circ W_\lambda^{(k+1)}(t;x,\Psi){\rm d}t\nonumber
\end{align*}
from which the statement of Proposition \ref{prop: LCFS-PR} follows. $\blacksquare$

\subsection{Proof of Theorem \ref{thm: main1}}\label{sec: main proof}
We assume here that the queues $\left\{\text{M}_t/\text{G}(\Psi)/1/k+\text{H}(\Psi)\right\}_{k\in\overline{\mathbb{Z}}_+}$ yielding the processes $\big\{W_\lambda^{(k)}(\cdot)\big\}_{k\in\overline{\mathbb{Z}}_+}$ are coupled so all the differences between them are only due to the assumption that they have different waiting room sizes, i.e., the next conditions are satisfied:
\begin{enumerate}
    \item In all of the queues $\left\{\text{M}_t/\text{G}(\Psi)/1/k+\text{H}(\Psi)\right\}_{k\in\overline{\mathbb{Z}}_+}$, the arrival process is the same.

    \item For every $i\in\mathbb{N}$, the $i$'th arriving customer (who is not necessarily joining) has the same patience and service times in all of the queues $\left\{\text{M}_t/\text{G}(\Psi)/1/k+\text{H}(\Psi)\right\}_{k\in\overline{\mathbb{Z}}_+}$. 
\end{enumerate}
Thus, for any fixed  $u>0$, $N_\lambda\left([0,u]\right)<\infty$, $\mathbb{P}$-a.s.  and hence there exists a $\mathbb{P}$-a.s. finite random variable $K$ such that $W_\lambda^{(k)}(t)=W_\lambda^{(\infty)}(t)$ and $\tau_\lambda^{(k)}\wedge u=\tau_\lambda^{(\infty)}\wedge u$ for every $t\in[0,u]$ and $k>K$, $\mathbb{P}$-a.s. where $a\wedge b\equiv\min\{a,b\}$ for any $a,b\in\mathbb{R}$.

Let $f:[0,\infty)\rightarrow[0,\infty)$ be a nondecreasing bounded continuous function on $[0,\infty)$. In addition, assume that $g:[0,\infty)\rightarrow[0,\infty)$ is continuous and bounded.  Then, a double application of dominated convergence theorem yields:
\begin{align}\label{eq: fixed u}
    \lim_{k\uparrow\infty} \mathbb{E}f\left[ \int_{0}^{\tau_{\lambda}^{(k)} \wedge u} g \circ W_{\lambda}^{(k)}(t){\rm d}t\right] &= \lim_{k\uparrow\infty} \mathbb{E}f\left[ \int_{0}^{u} g \circ W_{\lambda}^{(k)}(t) \textbf{1}_{(t,\infty)}\left(\tau_{\lambda}^{(k)}\right) {\rm d}t\right]\nonumber
    \\
    &= \mathbb{E}f\left[ \int_{0}^{u} g \circ W_{\lambda}^{(\infty)}(t) \textbf{1}_{(t,\infty)}\left(\tau_{\lambda}^{(\infty)}\right) {\rm d}t\right]\nonumber
    \\
    &= \mathbb{E}f\left[ \int_{0}^{\tau_{\lambda}^{(\infty)} \wedge u} g \circ W_{\lambda}^{(\infty)}(t) {\rm d}t\right]\,.
\end{align}
Consequently,
\begin{align}\label{eq: bound1}
    \mathbb{E}f\left[ \int_{0}^{\tau_{\lambda}^{(\infty)}} g \circ W_{\lambda}^{(\infty)}(t) {\rm d}t\right] &\overset{(\text{I})}{=} \lim_{u\uparrow\infty} \lim_{k\uparrow\infty} \mathbb{E}f\left[ \int_{0}^{\tau_{\lambda}^{(k)} \wedge u} g \circ W_{\lambda}^{(k)}(t){\rm d}t\right]\nonumber
    \\
    &\overset{(\text{II})}{\leq} \lim_{k\uparrow\infty} \mathbb{E}f\left[ \int_{0}^{\tau_{\lambda}^{(k)}} g \circ W_{\lambda}^{(k)}(t) {\rm d}t\right]\nonumber
    \\
    &\overset{(\text{III})}{\leq} \lim_{k\uparrow\infty} \mathbb{E}f\left[ \int_{0}^{\tau_{\lambda_h }^{(k)}} g \circ W_{\lambda_h }^{(k)}(t) {\rm d}t\right]\,.
\end{align}
Notice that: (I) stems from \eqref{eq: fixed u} and an application of the monotone convergence theorem. (II) is due to the assumptions that $f(\cdot)$ is nondecreasing and $g(\cdot)$ is non-negative. (III) follows from Proposition \ref{prop: LCFS-PR} since $f(\cdot)$ is nondecreasing.

The next step is to derive an upper bound for the limit that appears on the right-hand side of \eqref{eq: bound1}. To this end, we will need the following lemma.

\begin{lemma}\label{lemma: k bound}
 Assume that $g(\cdot)$ is a non-negative Borel function, $\Psi\in\mathcal{L}$ and $k,\ell\in\overline{\mathbb{Z}}_+$ such that $k\leq\ell$. Then, $A^{(k)}_{\lambda_h }(g;x,\Psi)\leq_{\normalfont \text{st}}A^{(\ell)}_{\lambda_h }(g;x,\Psi)$.
\end{lemma}
\begin{proof}
    Consider an $\text{M}/\text{G}(\Psi)/1/\ell+\text{H}(\Psi)$ queue and for every time $t\geq0$, let $L(t)$ be the number of customers who are present in the system at that time. Observe that 
    \begin{align}\label{eq: stochastic domination k}
    A^{(\ell)}_{\lambda_h }(g;x,\Psi)&= \int_{0}^{\tau_{\lambda_h }^{(\ell)}} g \circ W_{\lambda_h }^{(\ell)}(t){\rm d}t\\&\nonumber\geq\int_{0}^{\tau_{\lambda_h }^{(\ell)}} g \circ W_{\lambda_h }^{(\ell)}(t)\textbf{1}_{\left\{L(t)\leq k\right\}}{\rm d}t\,.
    \end{align}
    Now, recall that: (1) The service discipline is LCFS-PR. (2) The arrival process is homogeneous Poisson process with rate $\lambda_h $. Then, using the strong Markov property of the workload process with the memoryless property of the exponential distribution, we deduce that the lower bound on the right-hand side of \eqref{eq: stochastic domination k} is distributed like $A^{(k)}_{\lambda_h }(g;x,\Psi)$, from which Lemma \ref{lemma: k bound} follows. 
\end{proof}\newline\newline
Lemma \ref{lemma: k bound} (jointly with the standard coupling definition of $\leq_{\normalfont \text{st}}$) implies that there exists a probability space equipped with random variables $ \left\{A_k;k\in\overline{\mathbb{Z}}_+\right\}$ such that: (1) $0\leq A_0\leq A_1\leq A_2\leq \cdots\leq A_\infty$, $\mathbb{P}$-a.s. (2) $A_k$ is distributed like $A^{(k)}_{\lambda_h }(g;x,\Psi)$ for every $k\in\overline{\mathbb{Z}}_+$. Therefore, applying the monotone convergence theorem for this sequence shows that

\begin{align*}
\lim_{k\uparrow\infty} \mathbb{E}f\left[ \int_{0}^{\tau_{\lambda_h }^{(k)}} g \circ W_{\lambda_h }^{(k)}(t) {\rm d}t\right]&=\lim_{k\uparrow\infty}\mathbb{E}A_k=\mathbb{E}\sup_{k\in\mathbb{Z}_{\geq0}}A_k\\&\nonumber\leq \mathbb{E}A_\infty=\mathbb{E}f\left[ \int_{0}^{\tau_{\lambda_h }^{(\infty)}} g \circ W_{\lambda_h }^{(\infty)}(t) {\rm d}t\right]\,.   
\end{align*}
Thus, we have just proved that
\begin{equation}\label{eq: master}
\mathbb{E}f\left[ \int_{0}^{\tau_{\lambda}^{(\infty)}} g \circ W_{\lambda}^{(\infty)}(t) {\rm d}t\right]\leq \mathbb{E}f\left[ \int_{0}^{\tau_{\lambda_h }^{(\infty)}} g \circ W_{\lambda_h }^{(\infty)}(t) {\rm d}t\right]    
\end{equation}
for any nondecreasing bounded continuous function $f:[0,\infty)\rightarrow[0,\infty)$ and any bounded continuous function $g:[0,\infty)\rightarrow[0,\infty)$.

Let $f(\cdot)$ be continuous and bounded. Let $g:[0,\infty)\rightarrow[0,\infty)$ be a function for which there is a sequence $g_n:[0,\infty)\rightarrow[0,\infty)$ of continuous bounded functions such that $g_n\uparrow g$ pointwise as $n\to\infty$. Then, \eqref{eq: master} implies that 
\begin{equation*}
    \mathbb{E}f\left[ \int_{0}^{\tau_{\lambda}^{(\infty)}} g_n \circ W_{\lambda}^{(\infty)}(t) {\rm d}t\right]\leq \mathbb{E}f\left[ \int_{0}^{\tau_{\lambda_h }^{(\infty)}} g_n \circ W_{\lambda_h }^{(\infty)}(t) {\rm d}t\right]   
\end{equation*}
for every $n\in\mathbb{N}$. Now, an application of dominated convergence (letting the limit pass inside the expectation) and also monotone convergence (letting the limit pass inside the integral) shows that \eqref{eq: master} is valid for $g(\cdot)$.

Let $g(\cdot)$ be as above. For each $s\geq0$, define $f_s(y)\equiv \textbf{1}_{[s,\infty)}(y)$ for all $y\geq0$. For each $s\geq0$ we know that there is a sequence $f_{s,n}:[0,\infty)\rightarrow[0,1]$ of nondecreasing continuous functions such that $f_{s,n}\uparrow f_s$ pointwise as $n\to\infty$. Therefore, for each $s\geq0$ and $n\in\mathbb{N}$ we have  

\begin{equation}\label{eq: inequalityyyy}
    \mathbb{E}f_{s,n}\left[ \int_{0}^{\tau_{\lambda}^{(\infty)}} g \circ W_{\lambda}^{(\infty)}(t) {\rm d}t\right]\leq \mathbb{E}f_{s,n}\left[ \int_{0}^{\tau_{\lambda_h }^{(\infty)}} g \circ W_{\lambda_h }^{(\infty)}(t) {\rm d}t\right]\,.   
\end{equation}
Then, letting $n\to\infty$ from both sides of \eqref{eq: inequalityyyy} with the help of monotone convergence yields that
\begin{equation*}
    \mathbb{P}\left\{A_\lambda(g)\geq s\right\}\leq \mathbb{P}\left\{A_{\lambda_h }(g)\geq s\right\}
\end{equation*}
for every $s\geq0$.
Consequently, since  $A_\lambda(g)$ as well as $A_{\lambda_h }(g)$ are both non-negative random variables, Theorem \ref{thm: main1} follows. $\blacksquare$

\section{Proof of Theorem \ref{thm: super stochastic bound}}\label{sec: proof super bound}
For each $i\in\mathbb{N}$, let $I_\lambda^i\subseteq[0,\infty)$ (resp., $B_\lambda^i\subseteq[0,\infty)$) be the $i$'th idle (resp., busy) period in an $\text{M}_t/\text{G}(\Psi)/1+\text{H}(\Psi)$ queue. Define a random index 
\begin{equation*}
    \iota_\lambda\equiv\inf\left\{i\in\mathbb{N};|I_\lambda^i|>\kappa\right\}
\end{equation*}
and a random time
\begin{equation*}
    \zeta_\lambda\equiv\min\left\{\kappa n\in I_\lambda^{\iota_\lambda};W_\lambda(\kappa n)=0\ , \ n\in\mathbb{N}\right\}\,.
\end{equation*}
By construction, $\zeta_\lambda \geq \xi_\lambda$ and hence
\begin{align}\label{eq: bound11}
    \mathcal{A}_\lambda(g)=\int_0^{\xi_\lambda}g\circ W_\lambda(t){\rm d}t&\leq \int_0^{\zeta_\lambda}g\circ W_\lambda(t){\rm d}t\\&\nonumber=g(0)\sum_{i=1}^{\iota_\lambda}|I_\lambda^i|+\sum_{i=1}^{\iota_\lambda-1}\int_{B_\lambda^i}g\circ W_\lambda(t){\rm d}t\\&\nonumber\leq \iota_\lambda\kappa g(0)+\sum_{i=1}^{\iota_\lambda-1}\int_{B_\lambda^i}g\circ W_\lambda(t){\rm d}t
    \equiv\overline{\mathcal{A}}_\lambda(g)\,.
\end{align}
Similarly, observe that
\begin{align}\label{eq: bound2}
   \mathcal{A}^*_\lambda(g) = \int_0^{\xi_\lambda}g\circ W_\lambda(t-){\rm d}N^*_\lambda(t)&\leq \int_0^{\zeta_\lambda}g\circ W_\lambda(t-){\rm d}N^*_\lambda(t)\\&\nonumber\leq \iota_\lambda g(0)+\sum_{i=1}^{\iota_\lambda-1}\int_{B_\lambda^i}g\circ W_\lambda(t-){\rm d}N^*_\lambda(t)
    \equiv\overline{\mathcal{A}}^*_\lambda(g)\,.
\end{align}
Therefore, in order to prove Theorem \ref{thm: super stochastic bound}, it is sufficient to prove Proposition \ref{prop: bound1} and Proposition \ref{prop: bound2} below. Since the proofs of these propositions are almost identical, we will provide only that of Proposition \ref{prop: bound1}. 
\begin{proposition}\label{prop: bound1}
    Consider a probability space equipped with the following random variables:
    \begin{enumerate}
        %\item $\left(I^i_{\lambda_\ell}\right)_{i=1}^\infty$ is an iid sequence of random variables having an exponential distribution with rate $\lambda_\ell$.
    
    \item $\left(A^i_{\lambda_h}\right)_{i=1}^\infty$ is an iid sequence of random variables distributed like $A_{\lambda_h}(g;x,\Psi)$ with an initial state $x$ whose initial distribution is $G(\cdot)\equiv G(\Psi)(\cdot)$. 

    \item $\iota_{\lambda_h}$ is a random variable having a geometric distribution with probability of success equal to $e^{-\kappa\lambda_h}$.

    \item $\left(A^i_{\lambda_h}\right)_{i=1}^\infty$ and $\iota_{\lambda_h}$ are independent.  
    \end{enumerate}
    Then, 
    \begin{equation}\label{eq: stochastic bound}
        \mathcal{A}_\lambda(g)\leq_{\normalfont \text{st}}\overline{\mathcal{A}}_\lambda(g)\leq_{\normalfont \text{st}}\iota_{\lambda_h}  \kappa g(0)+\sum_{i=1}^{\iota_{\lambda_h}-1}A_{\lambda_h}^i\,.
    \end{equation}
\end{proposition}
\begin{proposition}\label{prop: bound2}
    Consider a probability space equipped with the following random variables:
    \begin{enumerate}
        %\item $\left(I^i_{\lambda_\ell}\right)_{i=1}^\infty$ is an iid sequence of random variables having an exponential distribution with rate $\lambda_\ell$.
    
    \item $\left(A^{i*}_{\lambda_h}\right)_{i=1}^\infty$ is an iid sequence of random variables distributed like $A^*_{\lambda_h}(g;x,\Psi)$ with an initial state $x$ whose initial distribution is $G(\cdot)\equiv G(\Psi)(\cdot)$. 

    \item $\iota^{*}_{\lambda_h}$ is a random variable having a geometric distribution with probability of success equal to $e^{-\kappa\lambda_h}$. 

    \item  $\left(A^{i*}_{\lambda_h}\right)_{i=1}^\infty$ and $\iota^*_{\lambda_h}$ are independent.   
    \end{enumerate}
    Then, 
    \begin{equation}\label{eq: stochastic bound11}
        \mathcal{A}^*_\lambda(g)\leq_{\normalfont \text{st}}\overline{\mathcal{A}}^*_\lambda(g)\leq_{\normalfont \text{st}}\iota_{\lambda_h}^*g(0)+\sum_{i=1}^{\iota^*_{\lambda_h}-1}A_{\lambda_h}^{i*}\,.
    \end{equation}
\end{proposition}

\subsection{Proof of Proposition \ref{prop: bound1}}
The inequality $\mathcal{A}_\lambda(g)\leq_{\normalfont \text{st}}\overline{\mathcal{A}}_\lambda(g)$ follows by construction. In addition, if $\lambda_h=0$, then $\iota_{\lambda_h}=\infty$, $\mathbb{P}$-a.s. and hence the upper bound in \eqref{eq: stochastic bound} is $\infty$, $\mathbb{P}$-a.s. and the claim follows trivially.
    
Let $\lambda_h>0$. Since $\lambda_h\geq\lambda(t)$ for every $t\geq0$, Theorem 1 from the work by Lewis and Shedler \cite{Lewis1979} implies that one can construct an inhomogeneous Poisson process  $N_{\lambda}(\cdot)$ with rate $\lambda(\cdot)$ which is a thinning of $N_{\lambda_h}(\cdot)$. For every $t\geq0$ let $D_{\lambda_h}(t)$ be the time of the first arrival according to $N_{\lambda_h}(\cdot)$ which occurs after $t$. Then, for each $i\in\mathbb{N}$, define
    \begin{equation*}
        I^i_{\lambda_h}\equiv D_{\lambda_h}\left(\inf I_\lambda^i\right)-\inf I_\lambda^i\,.
    \end{equation*}
    Note that, by construction, for every $i\in\mathbb{N}$, $\left|I_\lambda^i\right|\geq I^i_{\lambda_h}$ and hence, 
    \begin{equation*}   \iota_{\lambda_h}\equiv\inf\left\{i\in\mathbb{N};|I_{\lambda_h}^i|>\kappa\right\}\geq\iota_\lambda\ \ , \ \ \mathbb{P}\text{-a.s.}
\end{equation*}
Now, assume that $\left(U_i\right)_{i=1}^\infty$ is an iid sequence of random variables uniformly distributed on the unit interval. This sequence is independent from all other random variables which have already been introduced so far. Then, for each $i\in\mathbb{N}$, let $\upsilon^i_\lambda\equiv\inf B^i_\lambda$ and denote the service time of the customer who arrives at $\upsilon_\lambda^i$ by $\sigma_{\lambda}^i$ and note that $\left(\sigma_{\lambda}^i\right)$ are iid random variables having the distribution $G(\cdot)$. For every $(s,t)\in\mathbb{R}^2_+$, let $Q_\lambda(\cdot;s,t)$ (resp., $R_{\lambda_h}(\cdot;s)$) be the cdf of $A_{\lambda(t+\cdot)}(g;s,\Psi)$ (resp., $A_{\lambda_h}(g;s,\Psi)$) and for each $i\in\mathbb{N}$ define 
    \begin{equation*}
     Q_i(s,t)\equiv Q^{-1}_{\lambda(t+\cdot))}\left(U_i;s,\Psi\right)\ \ \left(\text{resp., } R_i(s)\equiv R^{-1}_{\lambda_h}\left(U_i;s,\Psi\right)\right).   
    \end{equation*}  
    Then,arguing as in the proof of Proposition \ref{prop: LCFS-PR}, Theorem \ref{thm: main1} implies that  
    \begin{equation*}
        Q_i\left(\sigma_\lambda^i,\upsilon_\lambda^i\right)\leq R_i\left(\sigma_\lambda^i\right)\ \ , \ \ \forall i\in\mathbb{N}\ , \  \mathbb{P}\text{-a.s.} 
    \end{equation*}
    Also, notice that $\left(R_i\left(\sigma_\lambda^i\right)\right)_{i=1}^\infty$ is an iid sequence of random variables distributed like $A_{\lambda_h}(g;x,\Psi)$ with a parameter $x$ distributed according to $G(\cdot)$. The Markov property of $W_\lambda(\cdot)$ implies that 
    \begin{equation*}
    \overline{\mathcal{A}}_\lambda(g)    \stackrel{\text{d}}{=}\iota_\lambda\kappa g(0)+\sum_{i=1}^{\iota_\lambda-1}Q\left(\sigma_\lambda^i,\upsilon_\lambda^i\right)
    \end{equation*}
    and so the result follows by setting $A_{\lambda_h}^i\equiv U\left(\sigma_\lambda^i\right)$ for each $i\in\mathbb{N}$. $\blacksquare$

\section{Proofs of Theorems \ref{thm: heavy-tail} and \ref{thm: heavy-tail11}}\label{sec: heavy tails}
%Theorem \ref{thm: stochastic bound} might be useful in order to study upper bounds of the convergence rate of $\mathbb{P}\left\{\mathcal{A}_\lambda(g)>u\right\}$ towards zero as $u\to\infty$. Specifically, 
%Assume \textcolor{blue}{(assume)} that $\rho_h\equiv\lambda_h\int_0^\infty y{\rm d}G(y)<1$, i.e., the system in which all customers join the queue is stable. 

Since $g(\cdot)$ is non-negative, Theorem \ref{thm: super stochastic bound} implies that
\begin{equation}\label{eq: heavy-tails-bound1}
    \mathbb{P}\left\{\mathcal{A}_\lambda(g)>u\right\}\leq\mathbb{P}\left\{\sum_{i=1}^{\iota_{\lambda_h}}\left[A_{\lambda_h}^i+g(0)\max\{\kappa,1\}\right]>u\right\}\ \ , \ \ \forall u>0\,,
\end{equation}
where the joint distribution of $\iota_{\lambda_h}$ and $\left(A^i_{\lambda_h}\right)_{i=1}^\infty$ was provided via the three conditions in the statement of Proposition \ref{prop: bound1}. In particular, the sum inside the probability which appears in the upper bound has a compound geometric distribution with probability of success $e^{-\lambda_h\kappa}$ and jumps that are distributed like $A_{\lambda_h}^1+\kappa g(0)$. 

For every $\Psi\in\mathcal{L}$, let $\Psi_{\infty}(\cdot)$ be the cdf of a bivariate random vector such that: (1) The first coordinate of the vector has the distribution $G(\Psi)(\cdot)$. (2) The second coordinate of the vector is equal to infinity $\mathbb{P}$-a.s. The next Lemma \ref{lemma: heavy-tailsss} will be useful and is easy to prove:
\begin{lemma}\label{lemma: heavy-tailsss}
For every $x>0$, $\Psi\in\mathcal{L}$ and nondecreasing function $g:[0,\infty)\rightarrow[0,\infty)$, 
\begin{equation*}
    A_{\lambda_h}(g;x,\Psi)\leq_{\normalfont\text{st}}
    A_{\lambda_h}(g;x,\Psi_\infty).
\end{equation*}
\end{lemma}

\subsection{Proof of Lemma \ref{lemma: heavy-tailsss}}
    Consider an $\text{M}_t/\text{G}(\Psi)/1+\text{H}(\Psi)$ queue. In the same probability space, construct an hypothetical queue with the same arrival process and customers, in which all customers, must join. Notice that the arrival epochs of the customers as well as their service times are the same in both queues, with the only difference that in the hypothetical queue, everyone joins. Consequently, the workload process of the hypothetical queue will never be below the workload process of the original $\text{M}_t/\text{G}(\Psi)/1+\text{H}(\Psi)$ queue and the result follows from this observation. $\blacksquare$\newline\newline 
 
    %Consider a stream of customers who arrive according to a homogeneous Poisson process with rate $\lambda_h$. Their service requirements are iid random variables having the distribution function $G(\cdot)\equiv G(\Psi)(\cdot)$. In addition, the customers also have patient times which are iid random variables having the distribution $H(\cdot)$. We assume that the arrival process, service times and patience times are all independent. 
    
    %Now, in that probability space there are two queues which are operated parallelly: (1) The first queue is such that any customers who arrives must join to. (2) The second queue is such that any customer who arrives joins to iff existing workload at his arrival time is not greater than his patience time. In both queues the service discipline is FCFS and the initial workload is equal to $x$.

    %Observe that at any moment the workload in the first queue is not less than the workload in the second queue. Thus, the result immediately follows from the assumption that $g(\cdot)$ is nondecreasing. $\blacksquare$
    \subsection{Proof of Theorem \ref{thm: heavy-tail} and Theorem \ref{thm: heavy-tail11} (continuation)}
Lemma \ref{lemma: heavy-tailsss} in conjunction with the inequality \eqref{eq: heavy-tails-bound1} imply that for any nondecreasing function $g:[0,\infty)\rightarrow[0,\infty)$,
\begin{equation*}
    \mathbb{P}\left\{\mathcal{A}_\lambda(g)>u\right\}\leq\mathbb{P}\left\{\sum_{i=1}^{\iota_{\lambda_h}}\left[\widetilde{A}_{\lambda_h}^i+ g(0)\max\{\kappa,1\}\right]>u\right\}\ \ , \ \ \forall u>0\,,
\end{equation*}
where:
\begin{enumerate}
    \item $\big(\widetilde{A}^i_{\lambda_h}\big)_{i=1}^\infty$ is an iid sequence of random variables distributed as $A_{\lambda_h}(g;x,\Psi_\infty)$ with an initial state $x$ which has an initial distribution  $G(\cdot)\equiv G(\Psi)(\cdot)$. 

    \item $\iota_{\lambda_h}$ is a random variable having a geometric distribution with probability of success equal to $e^{-\kappa\lambda_h}$.

    \item $\big(\widetilde{A}^i_{\lambda_h}\big)_{i=1}^\infty$ and $\iota_{\lambda_h}$ are independent.  
    \end{enumerate}
Observe that whenever $\widetilde{A}_{\lambda_h}^1$ has a subexponential distribution, $\widetilde{A}_{\lambda_h}^1+\kappa g(0)$ also has a subexponential distribution and hence Corollary 3 Embrechts, Goldie and Veraverbeke's paper \cite{Embrechts1979} implies that 
\begin{equation*}
    \lim_{u\to\infty}\frac{\mathbb{P}\left\{\sum_{i=1}^{\iota_{\lambda_h}}\left[\widetilde{A}_{\lambda_h}^i+g(0)\max\{\kappa,1\}\right]>u\right\}}{\mathbb{P}\left\{\widetilde{A}_{\lambda_h}^1>u-g(0)\max\{\kappa,1\}\right\}}=\frac{1-e^{-\kappa\lambda_h}}{e^{-\kappa\lambda_h}}\,.
\end{equation*}
Note that when $g(\cdot)$ is identically equal to one,  $\widetilde{A}_{\lambda_h}^1$ has the distribution of the length of a busy period. 

The proof of Theorem \ref{thm: heavy-tail11}, follows via an application of the same steps that made the proof of Theorem \ref{thm: heavy-tail}, up to very minor modifications. $\blacksquare$

\section{Proof of Theorem \ref{thm: eta}}\label{sec: proof eta}
Let $S_0$ be a random variable which is distributed according to $G(\cdot)$ and independent from the service times and patience times of the customers in the M/G/1+H queue with an arrival rate $\lambda_h$. 
By Theorem \ref{thm: super stochastic bound}, it will be sufficient to show that  
\begin{equation*}
\mu_{\lambda_h}\equiv\int_0^{\tau_{\lambda_h}\left(S_0\right)}{\rm d}N^*_{\lambda_h}(t)
\end{equation*}
has a finite $m$'th moment. To begin with, observe that once \eqref{eq: stability conditionnn} is valid, then there exists $\varepsilon>0$ such that 
\begin{equation}\label{eq: stability conditionn11}
    \lambda_h\int_0^\infty y{\rm d}G_\varepsilon(y)\lim_{y\uparrow\infty}\left[1-H(y)\right]=\lambda_h\left[\varepsilon+\int_0^\infty y{\rm d}G(y)\right]\lim_{y\uparrow\infty}\left[1-H(y)\right]<1,
\end{equation}
where $G_\epsilon(\cdot)\equiv G(\cdot-\varepsilon)$. In addition, notice that $G_\varepsilon(0)=0$ and 
\begin{equation*}
    \int_0^\infty y{\rm d}G_\varepsilon(y)=\varepsilon+\int_0^\infty y{\rm d}G(y)<\infty,
\end{equation*}
which means that the resulting M/$\text{G}_\varepsilon$/1+H queue is stable. 

For each $i$, let $S_i$ be the service time of the $i$'th arriving customer to the underlying M/G/1+H queue. In addition, we will refer to $S_0$ as the service time of the initial ($0$'th) customer who arrives at time $t=0$. On the same probability space, consider exactly the same queue with the only difference that now for each $i\in\mathbb{Z}_{\geq0}$, the service time of the $i$'th customer is equal to $S_i+\varepsilon$. Denote the workload process of that queue by $W_{\lambda_h,\varepsilon}(\cdot)\equiv W_{\lambda_h,\varepsilon}(\cdot;S_0+\varepsilon)$ and define
\begin{equation*}
    \tau_{\lambda_h,\varepsilon}\equiv\inf\left\{t\geq0; W_{\lambda_h,\epsilon}(t)=0\right\}\,.
\end{equation*}
The next lemma follows from the strong Markov property of the workload process. We will provide its proof in the sequel.

\begin{lemma}\label{lemma: epsilon}
    Let $g:[0,\infty)\rightarrow[0,\infty)$ be a Borel function. Then, for every $x>0$, 
    \begin{equation*}
        \int_0^{\tau_{\lambda_h}(S_0)}g\circ W_{\lambda_h}(t;S_0){\rm d}N^*_{\lambda_h}(t)\leq_{\normalfont \text{st}}\int_0^{\tau_{\lambda_h,\varepsilon}(S_0+\varepsilon)}g\circ W_{\lambda_h,\varepsilon}(t;S_0+\varepsilon){\rm d}N^*_{\lambda_h,\varepsilon}(t) \,.
    \end{equation*}
\end{lemma}
Lemma \ref{lemma: epsilon} implies that it is enough to show that
\begin{equation*}
    \mu_{\lambda_h,\varepsilon}\equiv\int_0^{\tau_{\lambda_h,\varepsilon}}{\rm d}N^*_{\lambda_h,\varepsilon}(t)
\end{equation*}
has a finite $m$'th moment. For that purpose, observe that
\begin{equation*}
0\leq \varepsilon\mu_{\lambda_h,\varepsilon}\leq \tau_{\lambda_h,\varepsilon}\,.
\end{equation*}
Therefore, if $m=1$ the claim follows since $\mathbb{E}\tau_{\lambda_h,\varepsilon}<\infty$, as explained in Application 3, in Section \ref{subsec: applications}. 

If $m\geq2$ and also $\int_0^\infty y^m{\rm d}G(y)<\infty$, then $\tau_{\lambda_h,\varepsilon}(x)$ has a finite $m$'th moment. This result is stated in the next lemma whose proof will be provided in the sequel. 
\begin{lemma}\label{lemma: moments}
    Fix $m\in\mathbb{N}_{\geq2}$. In addition, assume that \eqref{eq: stability conditionnn} is valid and $\int_0^\infty y^m{\rm d}G(y)<\infty$. Then, $\mathbb{E}\tau^m_{\lambda_h}(S_0)<\infty$. 
\end{lemma}
\subsection{Proof of Lemma \ref{lemma: epsilon}}
Consider the following recursive construction: $E_{0,\varepsilon}\equiv 0$ and for every $i\in\mathbb{N}$ define
\begin{align*}
&\overline{E}_{i,\varepsilon}\equiv\inf\left\{t\geq E_{i-1,\varepsilon};W_{\lambda_h,\varepsilon}(t)-W_{\lambda_h,\varepsilon}(t-)>0\right\}\,,\\& E_{i,\varepsilon}\equiv\inf\left\{t\geq \overline{E}_{i,\varepsilon};W_{\lambda_h,\varepsilon}\left(t\right)=W_{\lambda_h,\varepsilon}\left(\overline{E}_{i,\varepsilon}\right)-\varepsilon\right\}\,.\nonumber
\end{align*}
In words, $\overline{E}_{i, \varepsilon}$ is the arrival time of the $i$-th customer who joins the M/$\text{G}_\varepsilon$/1+H queue, and $E_{i+1, \varepsilon}$ is the first moment after $\overline{E}_{i, \varepsilon}$ when the workload in that queue decreases by $\varepsilon$ compared to its value at  $\overline{E}_{i, \varepsilon}$. Note that the choice of $\varepsilon$ is such that the M/$\text{G}_\varepsilon$/1+H queue is stable, i.e., $E_{i,\varepsilon}<\infty$ for every $i\in\mathbb{Z}_{\geq0}$.
Also, observe that by the strong Markov property of the workload process in the M/$\text{G}_\varepsilon$/1+H  queue,
\begin{equation*}
    \left\{W_{\lambda_h}(t;S_0);0\leq t\leq\tau_{\lambda_h}(S_0)\right\}\stackrel{\text{d}}{=} \left\{W_{\lambda_h,\epsilon}(t);t\in\mathcal{T}_{\lambda_h,\varepsilon}\right\}\,,
\end{equation*}
where
\begin{equation*}
    \mathcal{T}_{\lambda_h,\varepsilon}\equiv\left[0,\tau_{\lambda_h}(S_0)\right]\setminus\bigcup_{i=0}^\infty\left(E_{i,\varepsilon},\overline{E}_{i,\varepsilon}\right]\,.
\end{equation*}
Therefore, since $g(\cdot)$ is non-negative,  deduce that
\begin{align*}
\int_0^{\tau_{\lambda_h,\varepsilon}}g\circ W_{\lambda_h,\varepsilon}(t){\rm d}N^*_{\lambda_h}(t)&\geq \int_{\mathcal{T}_{\lambda_h,\varepsilon}}g\circ W_{\lambda_h,\varepsilon}(t){\rm d}N^*_{\lambda_h}(t)\\&\stackrel{\text{d}}{=}\int_0^{\tau_{\lambda_h}(S_0)}g\circ W_{\lambda_h}(t;S_0){\rm d}N^*_{\lambda_h}(t)\,,\nonumber
\end{align*}
from which Lemma \ref{lemma: epsilon} follows. $\blacksquare$

\subsection{The proof of Lemma \ref{lemma: moments}}
For $x\in[0,\infty)$ and $n\in\mathbb{Z}_{\geq0}$ denote $W_n\equiv W_{\lambda_h}(\kappa n;S_0)$ and $W_n(x)\equiv W_{\lambda_h}(\kappa n;x)$ . Notice that the sequence $\left(W_n\right)_{n=0}^\infty$ (resp., $\left(W_n(x)\right)_{n=0}^\infty$) is a homogeneous Markov chain with an initial state $S_0$ (resp., $x$) and state-space $[0,\infty)$. The next lemma will prove  helpful in the sequel where we provide the proof of Lemma \ref{lemma: moments}.
\begin{lemma}\label{lemma: rate}
For every $x>1$, 
\begin{equation*}
    \mathbb{E}W_1(x)-x+1\leq \lambda_h\int_0^\infty y{\rm d} G(y)\left[1-H(x-1)\right]\,.
\end{equation*}
\end{lemma}

% \textcolor{blue}{
% \begin{proof}
%     By the work of Lewis and Shedler, there exists a probability space and processes $N_1(.), N_2(.)$ defined on it which satisfies the following properties:
%     \begin{enumerate}
%         \item $N_1(.)$ is a Poisson process with a constant rate $\lambda_h \big[1-H(x-1)\big]$
%         \item $N_2(.)$ is a Poisson process with state-dependent rate $\lambda_h \big[1-H(W_{\lambda_h}(t))\big]$
%         \item $N_2(.)$ is a thinning of $N_1(.)$
%     \end{enumerate}
%     (3) is possible because on $t \in [0,1]$, $W_{\lambda_h}(t) \geq x-1$. This implies that on $t \in [0,1]$, $\lambda_h \big[1-H(x-1)\big] \geq \lambda_h \big[1-H(W_{\lambda_h}(t))\big]$. Let $\{T_i\}_{i \geq 1}$ (respectively $\{\Tilde{T}_i\}_{i \geq 1}$) denote the arrival times of non-balking customers to the system according to process $N_1(.)$ (respectively $N_2(.)$). Then, it follows that $T_j \leq_{\text{a.s.}} \Tilde{T}_j$ for all $j \geq 1$. Let $C_1$ (respectively $C_2$) denote the number of customers who effectively join the system in $[0,1)$ through process $N_1(.)$ (respectively $N_2(.)$). Then, $C_1 \geq_{\text{a.s.}} C_2$. Therefore, using the independence of $B_i$'s and $C_2$
%     \begin{align*}
%         \mathbb{E} W_1(x) -x +1 = \mathbb{E} \sum_{i=1}^{C_2} B_i \leq_{\text{a.s.}} \mathbb{E} \sum_{i=1}^{C_1} B_i = \lambda_h\int_0^\infty y{\rm d} G(y)\left[1-H(x-1)\right]
%     \end{align*}
% \end{proof}
% }

\begin{proof}
     Let $\eta_{\lambda_h}^*([0,1];x)$ be the number of jumps of the process $W_{\lambda_h}(\cdot;x)$ on the time interval $[0,1]$. In addition, for every $i\in\mathbb{N}$, let $S_i$ denote the service time of the $i$'th customer. Note that
     \begin{equation*}
        W_1(x)-x+1=\sum_{i=1}^{\eta_{\lambda_h}^*\left([0,1];x\right)} S_i\,. 
        \end{equation*}
        For each $i\in\mathbb{N}$, let $D_i$ (resp., $H_i$) be the inter-arrival time (resp., patience time) of the $i$'th customer. Also, let $D_0$ be the arrival time of the first customer. 
        
        Now, consider an additional queue defined in the same probability space as the original M/G/1+H queue. This new queue has the same initial workload and arrival process  as in the original queue. Furthermore, the customers in the new queue have the same patience times as well as service times as in the original queue. The difference is that in the new queue, for every $i\in\mathbb{N}$, the $i$'th customer joins iff $H_i>x-1$. Since the workload process in the original queue is not lower than $x-1$ during the time interval $[0,1]$, any customer who joins the original queue joins also the new queue. That is, if $\eta_{\lambda_h}^{**}\left([0,1];x\right)$ denotes the number of customers who join the new queue during the time interval $[0,1]$, then the inequality\begin{equation*}
        \eta_{\lambda_h}^*\left([0,1];x\right)\leq \eta_{\lambda_h}^{**}\left([0,1];x\right)
        \end{equation*}  holds pointwise. Consequently, since the service times are $\mathbb{P}$-a.s. non-negative, 

        \begin{equation*}
        \sum_{i=1}^{\eta_{\lambda_h}^*\left([0,1];x\right)} S_i\leq \sum_{i=1}^{\eta_{\lambda_h}^{**}\left([0,1];x\right)} S_i, \ \ \mathbb{P}\text{-a.s.}
        \end{equation*}
        In fact, the decision of the customers whether to join the new queue is invariant with respect to the sequence $\left(S_i\right)_{i=1}^\infty$ and hence Wald's identity implies that
        \begin{equation*}
        \mathbb{E}\sum_{i=1}^{\eta_{\lambda_h}^*\left([0,1];x\right)} S_i\leq \mathbb{E}\eta_{\lambda_h}^{**}\left([0,1];x\right)\int_0^\infty y{\rm d}G(y).
        \end{equation*}
        Note also that the decision of the customers whether to join  the new queue is determined solely by their patience time (independently of the workload process). Hence, the arrival process to the new queue is just a standard thinning of the arrival process which is a homogeneous Poisson process with rate $\lambda_h$. Therefore,
        \begin{equation*}
           \mathbb{E}\eta_{\lambda_h}^{**}\left([0,1];x\right)=\lambda_h\left[1-H(x-1)\right]\,, 
        \end{equation*}
        and the result follows. 
\end{proof}
\newline\newline
Now, we proceed with the proof of Lemma \ref{lemma: moments}. The stability condition \eqref{eq: stability conditionnn} is valid and hence we shall set $k\in(1,\infty)$ and $\delta\in(0,1)$ such that
\begin{equation*}
    \lambda_h\left[1-H(k-1)\right]\int_0^\infty y{\rm d}G(y)<1-\delta\,.
\end{equation*}
Denote $\mathcal{C}\equiv [0,k]$ and define the Lyapunov function
\begin{equation*}
    \psi(x)\equiv1+\frac{x}{\delta}\ \ , \ \ \forall x\geq0\,.
\end{equation*}
Notice that for every $x\in[0,\infty)$
\begin{equation}\label{eq: step 1}
    \mathbb{E}\psi\left(W_1(x)\right)\leq 1+\frac{x+\lambda_h\int_0^\infty y{\rm d}G(y)}{\delta(x_0)}<\infty\,.
\end{equation}
In addition, Lemma \ref{lemma: rate} implies that for every $x\in(1,\infty)$,  we have 
\begin{align}\label{eq: step 2}
    \mathbb{E}\psi\left(W_1(x)\right)-\psi\left(x\right)&\leq\frac{1}{\delta}\left\{\lambda_h\int_0^\infty y{\rm d}G(y)\left[1-H\left(k\right)\right]-1\right\}\\&\leq \frac{1-\delta-1}{\delta}=-1\,.\nonumber
\end{align}
Also, for $n\in\mathbb{N}$ let $L_n(x)$ (resp., $T_n(x)$) be the amount of input (resp., output) that was accumulated (resp., released) during the time interval $(n-1,n]$ when the initial state is $x$. Observe that for every $x\geq0$ we get
\begin{equation*}
\left|\psi\left(W_1(x)\right)-\psi\left(x\right)\right|\leq \frac{|L_1(x)|+|T_1(x)|}{\delta}= \frac{L_1(x)+1}{\delta}\,.
\end{equation*}
Note that $L_1(x)$ is uniformly bounded (in $x$) from above by the input when all customers ought to join the queue. This has a compound Poisson distribution with rate $\lambda_h$ and jump distribution $G(\cdot)$. It is known that this compound Poisson distribution has a finite $m$'th moment whenever $\int_0^\infty y^m{\rm d}G(y)<\infty$. Consequently,  
\begin{equation*}
    \label{eq: step 3}
\sup_{x\in[0,\infty)}\mathbb{E}\left|\psi\left(W_1(x)\right)-\psi\left(x\right)\right|^m<\infty\,.
\end{equation*}
Define the hitting time of the sequence $\left(W_n(x)\right)_{n=0}^\infty$ in the subset of states $\mathcal{C}$ by
\begin{equation*}
\tau_{\mathcal{C}}(x)\equiv\tau^{(1)}_{\mathcal{C}}(x)\equiv\inf\left\{n\in\mathbb{Z}_{\geq0};W_n(x)\in\mathcal{C}\right\}\,.    
\end{equation*}
According to Proposition 7.12\textbf{(i)} from the book by Bena\"im and Hurth \cite{Benaim2022}, since 
\eqref{eq: step 1}, \eqref{eq: step 2} and \eqref{eq: step 3}
hold, there exists $\alpha>0$ (uniform in $x$) such that for every $x\in\mathcal{C}$:
\begin{equation}\label{eq: recurrence}
    \mathbb{E}\tau_{\mathcal{C}}^m(x)\leq \alpha\left[1+\psi^m(x)\right]\leq \alpha\left[1+\left(1+\frac{k}{\delta}\right)^m\right]\,.
\end{equation}
Note that the upper bound in \eqref{eq: recurrence} is uniform in $x$ on $\mathcal{C}$.

In addition, define the successive return times of the sequence $\left(W_n(x)\right)_{n=0}^\infty$ in the subset of states $\mathcal{C}$, i.e., for each $n\in\mathbb{N}$ denote

\begin{equation*}
\tau^{(n+1)}_{\mathcal{C}}(x)\equiv \inf\left\{n>\tau^{(n)}_{\mathcal{C}}(x);W_n(x)\in\mathcal{C}\right\}\,. \end{equation*}
Then, for each $n\in\mathbb{Z}_{\geq0}$, let $Y_n(x)\equiv W_{\tau^{(n)}_{\mathcal{C}}(x)}(x)$ and notice that for every $x\in\mathcal{C}$ we have:
\begin{equation*}
\mathbb{P}\left\{\exists 1\leq i\leq k; Y_i=0\right\}\geq \mathbb{P}\left\{\text{No arrivals during} \ \left[0,k\right] \right\} \geq e^{-\lambda_hk}\,.
\end{equation*}
Thus, since $0\in\mathcal{C}$, Proposition 7.16\textbf{(ii)} in Bena\"im and Hurth's book \cite{Benaim2022} implies that $\mathbb{E}\xi_{\lambda_h}^m<\infty$. Lastly, $\mathbb{E}\tau_{\lambda_h}^m(S_0)<\infty$ since with positive probability, the time interval $[0,\xi_{\lambda_h}]$ contains at least one busy period whose length distribution is the same as the distribution of $\tau_{\lambda_h}(S_0)$ .  $\blacksquare$

\section{Discussion and future research}\label{sec: discussion}
Some examples of open problems that arise from the current research are discussed in the rest of this section.

\subsection{Stochastic monotonicity:}\label{subsec: stochastic monotonicity}
Here we formulate:
\begin{conjecture}
If $0\leq \lambda_1(t)\leq \lambda_2(t),\forall t\in[0,\infty)$, then
\begin{equation*}
  A_{\lambda_1}(g;x,\Psi)\leq_{\normalfont \text{st}}A_{\lambda_2}(g;x,\Psi), \ \forall x>0, \ g\in\mathcal{G}, \ \Psi\in\mathcal{L}\,.
\end{equation*}  
\end{conjecture}
Our Theorem \ref{thm: main1} includes an affirmative answer when $\lambda_2(\cdot)$ is a constant function. Actually, by very minor modifications of the proof of Theorem \ref{thm: main1}, one may get an affirmative answer also when $\lambda_1(\cdot)$ is a constant function. However, it is still an open problem whether this is true when both $\lambda_1(\cdot)$ and $\lambda_2(\cdot)$ are time-varying. 

In the current proof of Proposition \ref{prop: LCFS-PR}, the memoryless property of the exponential distribution plays an important role in the coupling construction. Specifically, we need to couple two {\normalfont  $\text{M}_t/\text{G}(\Psi)/1/k+\text{H}(\Psi)$} queues having, respectively, the arrival rates $\lambda_1(\cdot)$ and $\lambda_2(\cdot)$. Assume that $\lambda_1(t)\leq\lambda_2(t)=\lambda_h$ for every $t\geq0$ and in both queues the total service time of the initial customer is $x$. This allowed us to apply the thinning argument (Theorem 1 from the work by Lewis and Shedler \cite{Lewis1979})  \textit{only} on the service periods of the initial customer and argue that during all other times the arrivals are dictated by an external Poisson process with rate $\lambda_h$ independent of the arrival process during the service times of the initial customer. Due to the memoryless property of the exponential distribution (or equivalently, due to the strong Markov property of the homogeneous Poisson process), the resulting arrival process is still a Poisson process with rate $\lambda_2\equiv\lambda_h$. When $\lambda_2$ varies, we do not see how to adjust this construction of the arrival process to produce a Poisson arrival process with rate $\lambda_2(\cdot)$, i.e., the current proof may not be straightforwardly extended to the more general case.

\subsection{Stochastic $\lambda(\cdot)$}\label{subsec: stochastic lambda}
It is intriguing to think whether Theorems \ref{thm: main1} and \ref{thm: main2} are valid when $\lambda(\cdot)$ is a non-negative stochastic process such that $\lambda(t)\leq\lambda_h$ for every $t\geq0$, $\mathbb{P}$-a.s. If $\lambda$ is independent of the sequence $(S_1,Y_1),(S_2,Y_2),\ldots$, then Theorems \ref{thm: main1} and \ref{thm: main2} are easily extended by applying the same reasoning that leads to \eqref{eq: initial distribution} (namely, conditioning and un-conditioning on $\lambda(\cdot)$). A class of models belonging to this setup is based on a continuous-time Markov chain $J(\cdot)$ with countable state-space $\mathcal{J}$ independent of $(S_1,Y_1),(S_2,Y_2),\ldots$ such that $\lambda(t)\equiv\lambda_{J(t)}(t)$ where $\lambda_j(\cdot),j\in\mathcal{J}$ are all deterministic rate functions bounded from above by $\lambda_h$. Note that under the additional assumption that $\lambda_j(\cdot),j\in\mathcal{J}$ are all constant functions (not necessarily identical ones), then the arrival process will become the so called \textit{Markov modulated Poisson process} (see, e.g., the classic work by Neuts \cite{Neuts1979}).

It is also logical to think about \textit{state-dependent} $\lambda(\cdot)$, i.e., when $\lambda(t)$ is modulated by the state of the queue. Note that conditioning and un-conditioning on $\lambda(\cdot)$ is not useful anymore and other solution methods should be applied, e.g., one may try to generalize the current proof. 

For example, let $L(t)$ be the queue length process and assume that $\lambda(t)\equiv\lambda_{L(t)}(t)\leq\lambda_h$ for every $t\geq0$ (with $\lambda_\ell(\cdot)$ a deterministic rate function for every $\ell\in\mathbb{Z}_{\geq0}$). For this setup, one may verify that Proposition \ref{prop: LCFS-PR} remains valid, but then, when taking the size of the waiting room $k$ to infinity, a problem arises as the workload process in such a state-dependent queue is sensitive to the discipline of the service.

Another kind of potential state-dependence occurs under the assumption that $\lambda(\cdot)$ depends on the existing workload, i.e., if $T$ is distributed like the time
until the next arrival (starting from initial workload $x$), then 
\begin{equation*}
    \mathbb{P}\left\{T>t\right\}=\exp\left\{-\int_0^t\lambda\left[t,W_\lambda(t;x)\right]{\rm d}t\right\}\ \ , \ \ \forall t\in(0,\infty). 
\end{equation*}
For more details, see, e.g., the paper by Bekker \textit{et al.} \cite{Bekker2004}.
In this case, the workload process is invariant to the change of service discipline. However, the proof of Proposition \ref{prop: LCFS-PR} requires a generalization of the thinning arguement (i.e., Theorem 1 in the work by Lewis and Shedler \cite{Lewis1979}) to state-dependent Poisson arrival processes. To the best of our knowledge, extensions in this direction (of Theorem 1 in the work by Lewis and Shedler \cite{Lewis1979}) do not appear in the existing literature and, respectively, we consider its derivation as an interesting open question for future research. 

\subsection{Bounds on $\mathbb{E}g\left[W_\lambda(\infty)\right]$} 
Consider the periodic  $\text{M}_t/\text{G}(\Psi)/1+\text{H}(\Psi)$ queue discussed in Section \ref{subsec: applications} with the notations presented in the proof of Theorem \ref{thm: super stochastic bound} and for simplicity assume that $g\in\mathcal{G}$ is such that $g(0)=0$. The regenerative property of the workload process yields, that when the queue is stable, we have the following identity of the steady-state workload:
\begin{equation}\label{eq: moment bound}
    \mathbb{E}g\left[W_\lambda(\infty)\right]=\frac{\mathbb{E}\mathcal{A}_\lambda(g)}{\mathbb{E}\xi_\lambda}\,.
\end{equation}
Note that any customer who arrives at an empty queue will join the queue. Thus, under the assumption that the $\text{M}/\text{G}(\Psi)/1+\text{H}(\Psi)$ queue with an arrival rate $\lambda_h$ is also stable and $g(0)=0$, the regenerative property of the workload process yields: 
\begin{equation*}
    \mathbb{E}g\left[W_{\lambda_h}(\infty)\right]=\frac{\mathbb{E}A_{\lambda_h}(g)}{\frac{1}{\lambda_h}+\mathbb{E}\tau_{\lambda_h}}.
\end{equation*}
It might be interesting to derive bounds for $\mathbb{E}g\left[W_{\lambda}(\infty)\right]$ in terms of $\mathbb{E}g\left[W_{\lambda_h}(\infty)\right]$. In practice, such a bound might be useful, e.g., when $\Psi$ has a product-form, because then the distribution of $W_{\lambda_h}(\infty)$ can be figuered out (at least for special cases) as described by Baccelli, Boyer, and Hebuterne \cite{Baccelli1984}.

To see, how such a bound may be constructed, recall the following result
\begin{equation*}
    \mathcal{A}_\lambda(g)\leq_{\normalfont \text{st}}\sum_{i=1}^{\iota_{\lambda_h}}A_{\lambda_h}^i,
\end{equation*}
which stems from Proposition \ref{prop: bound1}. In addition, assume that $\lambda(t)\geq\lambda_\ell$ for every $t\geq0$ for some $\lambda_\ell>0$. Since any customer who arrives to an empty queue, joins the queue, by construction (with the help of the strong Markov property) deduce that 
\begin{equation*}
\xi_\lambda\geq_{\text{st}}\tau_{\lambda(\cdot+E)}\textbf{1}_{[0,\kappa]}(E),  
\end{equation*}
where:
\begin{itemize}
    \item $E$ is a random variable distributed like the arrival epoch of the first customer in the $\text{M}_t/\text{G}(\Psi)/1+\text{H}(\Psi)$ queue.  Note that as $\lambda(t)\geq\lambda_\ell$ for every $t\geq0$, we have:
    \begin{equation*}
        \mathbb{P}\left\{E\leq\kappa\right\}\geq1-e^{-\lambda_\ell\kappa}.
    \end{equation*}

    \item For every $e\geq0$, $\tau_{\lambda(\cdot+e)}$ is distributed as $\tau_{\widetilde{\lambda}}$ such that $\widetilde{\lambda}(t)=\lambda(t+e)$ for every $t\geq0$.

    \item The random variables belonging to the collection $\{\tau_{\lambda(\cdot+e)};e>0\}\cup\{E\}$ are mutually independent and also independent of the $\text{M}_t/\text{G}(\Psi)/1+\text{H}(\Psi)$ queue.

\end{itemize} 
Now, since  $\iota_{\lambda_h}\sim\text{Geo}(e^{-\lambda_h\kappa})$, by combining these results all together, deduce that
\begin{align*}
\mathbb{E}g\left[W_{\lambda}(\infty)\right]&\leq e^{\lambda_h\kappa}\cdot\frac{\frac{1}{\lambda_h}+\mathbb{E}\tau_{\lambda_h}}{\mathbb{P}\left\{E\leq\kappa\right\}\mathbb{E}\left[\tau_{\lambda(\cdot-E')}\big|E\leq\kappa\right]}\cdot\mathbb{E}g\left[W_{\lambda_h}(\infty)\right]\\&\leq \frac{e^{\lambda_h\kappa}}{1-e^{-\lambda_\ell\kappa}}\cdot\frac{\frac{1}{\lambda_h}+\mathbb{E}\tau_{\lambda_h}}{\mathbb{E}\left[\tau_{\lambda(\cdot-E')}\big|E'\leq\kappa\right]}\cdot\mathbb{E}g\left[W_{\lambda_h}(\infty)\right]\\&\nonumber\leq \frac{e^{\lambda_h\kappa}}{1-e^{-\lambda_\ell\kappa}}\cdot\left(1+\rho_h^{-1}\right)\mathbb{E}g\left[W_{\lambda_h}(\infty)\right],    
\end{align*}
where $\rho_h\equiv\lambda_h\int_0^\infty x{\rm d} G(x)\in(0,\infty)$ and we applied: (1) Theorem \ref{thm: main1} (see, also the discussion in Section \ref{subsec: stochastic lambda}) in order to bound the ratio $\frac{\mathbb{E}\tau_{\lambda_h}}{\mathbb{E}[\tau_{\lambda(\cdot-E)}|E\leq\kappa]}$ from above by one. (2) $\tau_{\lambda(\cdot-E)}$ is not less than the service time of the first customer having distribution $G(\cdot)\equiv G(\Psi)(\cdot)$.  

Primarily, by setting $g(w)=w^k$ for some $k\in\mathbb{N}$, the proof that the $k$'th moment of the steady-state distribution of $W_\lambda(\cdot)$ is finite reduces to showing that the $k$'th moment of the steady-state distribution of $W_{\lambda_h}(\cdot)$ is finite. 

For future research, it might be interesting to try to evaluate the quality of this bound, i.e., to see whether there are tighter bounds. In addition, it is also interesting to look for lower bounds. %Potentially, for this purpose, recall that as mentioned in Section \ref{subsec: stochastic monotonicity}, it is possible to derive the analogue of Theorem \ref{thm: main1} under the assumption that $\lambda_\ell\leq\lambda(t)$ for every $t\geq0$.    

\subsection{Tail distribution of $\mathcal{A}_\lambda(g)$ and $\mathcal{A}^*_\lambda(g)$}
Theorem \ref{thm: heavy-tail} (resp., Theorem \ref{thm: heavy-tail11}) indicates some pre-conditions on $G(\cdot)$ under which it states an upper bound on the rate at which $\mathbb{P}\left\{\mathcal{A}_\lambda(g)>u\right\}$ (resp., $\mathbb{P}\left\{\mathcal{A}^*_\lambda(g)>u\right\}$) vanishes as $u\to\infty$ for the special case when $g(w)=1$ for all $w\geq0$. The following questions arise:
\begin{enumerate}
    \item Is it possible to derive similar bounds for other choices of $g(\cdot)$? 

    \item If $g(w)=1$ for every $w\geq0$, is it possible to derive a lower bound on the rates at which $\mathbb{P}\left\{\mathcal{A}_\lambda(g)>u\right\}$ and  $\mathbb{P}\left\{\mathcal{A}^*_\lambda(g)>u\right\}$ vanish as $u\to\infty$?

    \item If $g(w)=1$ for every $w\geq0$, what can be said regarding the tightness of the upper bounds stated in Theorem \ref{thm: heavy-tail} and Theorem \ref{thm: heavy-tail11}? 
\end{enumerate}

\subsection{Multi-server queues:} 
%\textcolor{blue}{Michel mentioned to explain in more detail why the methodology does not work for multi-server queues. One reason is also that the virtual waiting time explicitly defines state of the system in single server models but not in multi-server models.}
Consider the $\text{M}_t/\text{G}(\Psi)/s+\text{H}(\Psi)$ queue, i.e., exactly the same model as the $\text{M}_t/\text{G}(\Psi)/1+\text{H}(\Psi)$ except that now we have $s\geq2$ servers and respectively the decision whether to join or balk is based on the virtual waiting time process instead of the workload. More specifically, at each arrival epoch, the arriving customer observes his future waiting time in each of the $s$ waiting lines (corresponding to the $s$ servers). He compares the minimal future waiting time with his patience level and joins that waiting line if and only if his patience time is not less than the observed minimum. Otherwise, he should balk. 

We do not see how the current methodology can be applied to extend Theorems \ref{thm: main1} and \ref{thm: main2} to the multi-server queue even when $\Psi(\cdot)$ has a product-form. In general, the joining/balking rule of the customers depends on the virtual waiting time process whose precise mathematical description (when $\Psi(\cdot)$ has product form) appears in, e.g.,  Section 2 of the work by Bodas, Mandjes and Ravner \cite{Bodas2023}. Notably, when the queue has a single server, the virtual waiting time under FCFS becomes the workload process and hence it is possible to phrase the joining/balking rule in terms of the workload process. This enabled us to prove Proposition \ref{prop: LCFS-PR} via an inductive argument which however fails whenever there is more than one server. 

Some existing results for the first moment and some specific choices of $g(\cdot)$ in the multi-server queue have been derived in the recent work by Bodas, Mandjes and Ravner \cite{Bodas2023} by applying Lyapunov functions. We are intrigued by the possibility of applying other techniques in order to extend the stochastic bounds appearing in the present work to the more general multi-server queueing model. 

\subsection{Observable queues:}
In the $\text{M}_t/\text{G}(\Psi)/1+\text{H}(\Psi)$ queue, each customer observes the existing workload at his arrival epoch. This assumption is inconsistent with some queues that we know from the daily life, in which each customer observes the queue length (but not the workload) at his arrival epoch. A discussion regarding such queues (when $\Psi(\cdot)$ has product-form) appears in Section 6 in the recent work by Bodas, Mandjes and Ravner \cite{Bodas2023}. In particular, the analogue of Theorem \ref{thm: eta} in that model with an observable queue is motivated by the discussion which appears above Proposition 6.1 in \cite{Bodas2023}. 

We admit that although we tried, in that model it is impossible to apply the inductive argument in the core of the proof of Proposition \ref{prop: LCFS-PR}. Consequently, the current methodology is not applicable in for deriving the analogue results under this new assumption regarding the customer's information at their arrival epochs.  

\subsection{More applications of the current methodology:} 
In the present work we applied a methodology which is based on the following steps:
\begin{enumerate}
    \item Restrict the size of the waiting room and transfer the service discipline to LCFS-PR.

    \item Couple the two queues under the above-mentioned waiting-room restriction.

    \item Prove the required stochastic bound for every finite waiting room via induction.

    \item Take the size of the waiting-room to infinity and get the desired stochastic bound. 
\end{enumerate}

We are wondering about existence of other interesting stochastic bounds for other queues that can be derived by applying similar ideas. Especially, it might be interesting to consider the same queueing model, but with different joining/balking rules. For example, we anticipate that similar arguments to those used here should lead to similar results if we allow the arriving customers to know their service times at their arrival epochs, i.e., each one of them balks iff his patience time is less than the existing workload (\textit{including} his service) at his arrival epoch. 

We also note in passing that we are optimistic regarding potential application of the presented methodology in order to derive similar results in a model with no waiting room and orbiting customers having exponential orbiting times (see, e.g., the survey by Yang and Templeton \cite{Yang1987} regarding retrial queues).  It is also intriguing to see whether the same methodology is applicable to similar queues with infinite waiting room but finite workload capacity. 
\newline\newline
\textbf{Acknowledgment:} We are indebted to Onno Boxma, Yan Dolinsky, Purva Joshi, Michel Mandjes, Ohad Perry, Liron Ravner and Danny Segev for useful discussions and practical comments on preliminary versions of the present work.

\end{document}